\documentclass[11pt,reqno]{amsart}%
\usepackage[latin1]{inputenc}
\usepackage{amsfonts,amssymb,latexsym}
\usepackage{amsmath,amsthm}
\usepackage[Sonny]{fncychap}
\usepackage{enumerate}
\usepackage[dvips]{color,graphicx}
\usepackage{amsbsy}
\usepackage{amsfonts}
\usepackage{graphicx}
\usepackage{color}
\usepackage{amsmath}
\usepackage{amssymb}%
\newtheorem{theorem}{Theorem}[section]
\newtheorem{definition}[theorem]{Definition}
\newtheorem{proposition}[theorem]{Proposition}

\newtheorem{lemma}[theorem]{Lemma}
\newtheorem{remark}{Remark}[section]

\numberwithin{equation}{section}
\begin{document}
\title[ ]{On the non-homogeneous Navier-Stokes system with Navier friction boundary conditions}
\author[L. C. F. Ferreira]{Lucas C. F. Ferreira}
\author[G. Planas]{Gabriela Planas}
\address[Lucas C. F. Ferreira and Gabriela Planas]{Departamento de Matem{\'a}tica,
Instituto de Matem\'atica, Estat\'\i stica e Computa\c{c}\~ao Cient\'\i fica,
Universidade Estadual de Campinas, Rua Sergio Buarque de Holanda, 651,
13083-859, Campinas-SP, Brazil}
\email[L.C.F. Ferreira and G. Planas]{lcff@ime.unicamp.br and gplanas@ime.unicamp.br}
\author[E. J. Villamizar-Roa]{Elder J. Villamizar-Roa }
\address[Elder J. Villamizar-Roa]{ Universidad Industrial de Santander, Escuela de
Matem\'{a}ticas, A.A. 678, Bucaramanga, Colombia}
\email{jvillami@uis.edu.co}
\thanks{L.C.F. Ferreira was supported by FAPESP and CNPq - Brazil; G. Planas has been
partially supported by FAPESP - Brazil, grant 2007/51490-7 and CNPq - Brazil,
grant 303302/2009-7}
\date{\today}

\begin{abstract}
We address the issue of existence of weak solutions for the non-homogeneous
Navier-Stokes system with Navier friction boundary conditions allowing the
presence of vacuum zones and assuming rough conditions on the data. We also
study the convergence, as the viscosity goes to zero, of weak solutions for
the non-homogeneous Navier-Stokes system with Navier friction boundary
conditions to the strong solution of the Euler equations with variable
density, provided that the initial data converge in $L^{2}$ to a smooth enough limit.

\end{abstract}
\maketitle

\section{Introduction}

We are concerned with the incompressible Navier-Stokes model with variable
density in a bounded domain. The governing equations are given by the
following system
\begin{equation}
\left\{
\begin{array}
[c]{rcl}%
\partial_{t}(\rho u)+\mbox{div}\ (\rho uu)-\nu\Delta u+\nabla\pi=\rho
f\ \mbox{in}\ Q, &  & \\
\mbox{div}\ u=0\ \mbox{in}\ Q, &  & \\
\partial_{t}\rho+\mbox{div}\ (\rho u)=0\ \mbox{in}\ Q. &  &
\end{array}
\right.  \label{prev0}%
\end{equation}
Here, $Q\equiv\Omega\times(0,T),$ where $\Omega$ is a bounded domain of
$\mathbb{R}^{3}$ with smooth boundary $\partial\Omega,$ and $T>0$. The
unknowns are the velocity field $u$, the density $\rho$, and the pressure
$\pi$ of the fluid. The parameter $\nu>0$ is the viscosity coefficient of the
fluid and $f$ is a given vector field driving the motion.

We supplement the system \eqref{prev0} with initial and Navier friction
boundary conditions
\begin{equation}
\left\{
\begin{array}
[c]{rcl}%
u\cdot n=0\ \mbox{on}\ \Sigma, &  & \\
\left[  D(u)n+\alpha u\right]  _{tan}=0\ \mbox{on}\ \Sigma, &  & \\
\rho(0)=\rho_{0}\ \mbox{in}\ \Omega, &  & \\
(\rho u)(0)=v_{0}\ \mbox{in}\ \Omega, &  &
\end{array}
\right.  \label{prev0b}%
\end{equation}
where $\Sigma\equiv\partial\Omega\times(0,T),$ $n$ is the exterior normal
vector to $\partial\Omega$, $\rho_{0}\geq0$ denotes the initial density and
$v_{0}$ has to be at least such that $v_{0}(x)=0$ whenever $\rho_{0}(x)=0.$
Moreover, $D(u)=\frac{1}{2}(\partial_{i}u_{j}+\partial_{j}u_{i})_{1\leq
i,j\leq n}$ denotes the deformation tensor, $[\cdot]_{tan}$ is the tangential
component of a vector on $\partial\Omega,$ and $\rho uu=\rho(u\otimes u)$. The
constant $\alpha\geq0$ stands for the friction coefficient which measures the
tendency of the fluid to slip on the boundary.

The goal of this paper is to study the convergence of solutions for
(\ref{prev0})-(\ref{prev0b}), as the viscosity goes to zero, toward the solution
of the Euler equations with variable density. Formally, when we drop the
viscous term (i.e., taking $\nu=0$) system \eqref{prev0}-\eqref{prev0b}
degenerates into the non-homogeneous Euler equations
\begin{equation}
\left\{
\begin{array}
[c]{rcl}%
\partial_{t}(\rho u)+\mbox{div}\ (\rho uu)+\nabla\pi=\rho f\  &
\mbox{in}\ Q, & \\
\mbox{div}\ u=0\  & \mbox{in}\ Q, & \\
\partial_{t}\rho+\mbox{div}\ (\rho u)=0\  & \mbox{in}\ Q, & \\
u\cdot n=0\  & \mbox{on}\ \Sigma, & \\
\rho(0)=\rho_{0}\  & \mbox{in}\ \Omega, & \\
(\rho u)(0)=v_{0}\  & \mbox{in}\ \Omega. &
\end{array}
\right.  \label{euler}%
\end{equation}
We aim at giving a justification of this formal procedure.

The issue of the vanishing viscosity limit or inviscid limit for the
incompressible homogeneous Navier-Stokes equations is a classical problem in
fluid mechanics. In the whole space and periodic cases, the inviscid limit was
performed by several authors, see e.g. \cite{ConstantinFoias,Constantin,
Kato,Yudovich}. In the case where there exist physical boundaries, the problem
of convergence leads to the formation of a boundary layer if one supplements
the Navier-Stokes equations with no-slip boundary conditions (which are the
most often prescribed ones). This happens because there is a discrepancy
between the no-slip boundary conditions for the Navier-Stokes equations and
the tangential boundary conditions for the Euler equations.

There is no consensus on the boundary conditions to be prescribed for the
Navier-Stokes equations, except for impermeable boundary which corresponds to
the condition \eqref{prev0b}$_{1}$. Navier \cite{Navier} claimed that the
tangential component of the viscous stress at the boundary should be
proportional to the tangential velocity, leading to the boundary condition
\eqref{prev0b}$_{2}$. Conditions \eqref{prev0b}$_{1}$-\eqref{prev0b}$_{2}$ are
called Navier friction boundary conditions, or simply Navier boundary
conditions (another names have been used as well). These conditions were also
derived by Maxwell \cite{Maxwell} from the kinetic theory of gases and
rigorously justified as a homogenization of the no-slip condition on a rough
boundary (see \cite{Jager}).

Recently, the inviscid limit for the Navier-Stokes equations with Navier
boundary conditions was established, for which the reader is referred to
\cite{Clopeau,Iftimie-Planas,Lopes,Masmoudi}. The situation in this case is
thus very different from the case of no-slip boundary conditions and requires
distinct involved arguments.

Concerning the non-homogeneous incompressible Navier-Stokes equations, it is
worthwhile to remark that there exists a considerable number of papers devoted
to their mathematical analysis, especially in the case where the equations are
complemented with Dirichlet boundary conditions. Those results can be
classified in two classes: on the one hand, there are existence results when
the initial density is assumed to be positive and so there is no vacuum
initially; and on the other hand, the case where the initial-vacuum is
allowed. The first case has been addressed by several authors, see e.g.
\cite{Antokazhi, Danchin1, kazhi, Lady, Okamoto, salvi, BRV,itoh-tani}, and
references therein. In order to avoid vacuum, the basic assumption in the
above-quoted works is
\[
0<c_{0}\leq\inf_{x\in\Omega}\rho_{0}(x)\leq\rho_{0}(x)\leq\sup_{x\in\Omega
}\rho_{0}(x),
\]
and so, in particular $\rho_{0},\rho$ have a positive lower bound. In the
second one, when the initial-vacuum is allowed $(\rho_{0}\geq0)$, the problem
\eqref{prev0}-\eqref{prev0b} is more difficult to handling. Indeed, comparing
with the first case, fewer results are available in the literature related to
the existence of weak solutions (see \cite{Simon,Simonb,Kim,lions}). In
particular, to the best of our knowledge, existence of weak solutions has not
been still treated for the non-homogeneous incompressible Navier-Stokes system
with Navier boundary conditions. Let us mention the work \cite{itoh2000} where
strong solutions to the system with slip boundary conditions are considered,
however vacuum zones are not admitted. Thus, our first goal will be show the
existence of weak solutions for \eqref{prev0}-\eqref{prev0b}, allowing vacuum
and assuming rough conditions on the field $f$ and initial data $\rho_{0}$ and
$v_{0}.$

The vanishing viscosity limit for the non-homogeneous incompressible
Navier-Stokes equations in the whole space, or with periodic boundary
conditions, was addressed in \cite{Itoh,itoh-tani,Danchin}, as long as no
vacuum states occur. They proved the convergence of local strong solutions in
Hilbert and Sobolev spaces. One of the difficulties in this case is to show
that the time existence is independent of the viscosity. As well as in the
homogeneous case, it is expected that the issue of the inviscid limit in
bounded domains presents boundary-layer phenomenon when one considers no-slip
boundary condition.

Our second goal in this paper is to show that a weak solution of the
non-homogenous incompressible Navier-Stokes equations \eqref{prev0} with
Navier boundary conditions \eqref{prev0b} converges in the energy space toward
the strong solution of the non-homogeneous incompressible Euler equations
\eqref{euler} in the inviscid limit. This extends some earlier results
obtained by \cite{Clopeau,Lopes,Iftimie-Planas,Paddick} in the homogeneous
case. The strategy is to compare the smooth solution of the Euler equations
and a weak solution of the Navier-Stokes equations, which is the leading idea
of the proof of weak-strong uniqueness for the homogeneous incompressible
Navier-Stokes equations. In our case, this approach arises suitably in view of
the fact that only strong solutions are known to exist for the non-homogeneous
Euler equations (\ref{euler}) (see e.g. \cite{valli88}). In fact, to best of
our knowledge, there is no theory of weak solutions for (\ref{euler}).

This paper is structured as follows. In the next section, we establish a
result of existence of weak solutions with finite energy for the problem
\eqref{prev0}-\eqref{prev0b} (see Theorem \ref{teorem1}). The precise
definition of weak solution with finite energy is given in Definition
\ref{definicion1} below.
%We remark that when the initial density
%is positive, the solution obtained is weak-strong.
In Section \ref{inviscid}, we address the vanishing viscosity limit for the
non-homogeneous Navier-Stokes equations with Navier friction boundary
conditions (see Theorem \ref{limitinv}). To this end, we first recall a result
of local strong solution for the Euler equations with variable density (see
Theorem \ref{euler1}). Finally, we proceed with the proof of the inviscid limit.

We finish this section by establishing some notations used throughout this
manuscript. We denote by $\mathcal{D}(\Omega)$ and $\mathcal{D}^{\prime
}(\Omega)$ the space of functions of class $C^{\infty}(\Omega)$ with compact
support, and the space of distributions on $\Omega,$ respectively. We use
standard notations for Lebesgue and Sobolev spaces. We denote by $\Vert
\cdot\Vert_{p}$ the norm in $L^{p}(\Omega)$. Otherwise, the norm will be
specified. For a Banach space $X$, we indicate by $\langle\cdot,\cdot
\rangle_{X^{\prime},X}$ the duality product between $X^{\prime}$ (the dual
space of $X$) and $X$. As usual, we will use the same notation for vector
valued and scalar valued spaces. There will be no danger of confusion since
the difference will be clear in the context.
Also, we denote by $H_{\sigma}^{1}(\Omega)$ the subspace of $H^{1}(\Omega)$ of
divergence free vector fields tangent to the boundary. Finally, $C_{w}%
([0,T];X)$ represents the space of functions $u:[0,T]\rightarrow X$ which are
continuous with respect to the weak topology.

%We introduce the spaces of divergence free vector fields
%\[
%H=\{v\in(L^2(\Omega))^{3}:\mbox{div}\ v=0 \mbox{ in } \Omega, \ v \cdot n = 0 \mbox{ on } \partial \Omega\}
%\]
%and
%\[
%V=\{v\in(W^{1,2}(\Omega))^{3}:\mbox{div}\ v=0 \mbox{ in } \Omega, \ v \cdot n = 0 \mbox{ on } \partial \Omega\}.
%\]

\section{Weak solutions for the non-homogeneous Navier-Stokes system}

In this section we study the existence of global weak solutions for
the non-homogeneous incompressible Navier-Stokes with Navier
boundary conditions. We assume that the initial density $\rho_{0}$
belongs to $L^{p}(\Omega),$ $6\leq p\leq\infty,$ allowing to vanish,
$v_{0}\in L^{\frac{2p}{p+1}}(\Omega)$ with
$\frac{|v_{0}|^{2}}{\rho_{0}}\in L^{1}(\Omega)$ and the external
force $f\in L^{1}(0,T;L^{\frac{2p}{p-1}}(\Omega)).$ The data $v_{0}$ and $\frac{|v_{0}|^{2}}{\rho_{0}}$ correspond formally to initial value for $\rho u$ and $\rho\left\vert u\right\vert ^{2},$
respectively. In the case $p=\infty,$ $L^{r}$-exponents depending on $p$
should be understood in the natural way, that is, as the limit when
$p\rightarrow\infty$. For instance, $L^{\frac{2p}{p+1}}$ and $L^{\frac
{2p}{p-1}}$ become $L^{2}$ in that case.

Now we introduce the definition of weak solution with finite energy for the
system \eqref{prev0}-\eqref{prev0b}.

\begin{definition}
\label{definicion1}

A weak solution for \eqref{prev0}-\eqref{prev0b} is a pair of functions
$(u,\rho)$ verifying the following items:

\begin{enumerate}
\item[i)] $u\in L^{2}(0,T;H^{1}_{\sigma}(\Omega)),$ $\rho\in C([0,T];W^{-1,p}
(\Omega))\cap L^\infty(0,T;L^p(\Omega)),$ $\rho\geq0$ a.e. in $Q,$
${\rho u}\in L^{\infty}(0,T;L^{\frac{2p}{p+1}}(\Omega)),$
$\sqrt{\rho} u\in L^{\infty}(0,T;L^{2 }(\Omega)),$ $\rho
uu\in L^1(0,T;L^2(\Omega)),$ such that equation
$\partial_{t}\rho+\mbox{div}\ (\rho u)=0$ is satisfied in
$\mathcal{D}^{\prime}(Q)$ and the momentum equation
(\ref{prev0})$_{1}$ is verified in the following sense:
\begin{align}
-\int_{0}^{T}\int_{\Omega}\rho u\cdot\partial_{t}\varphi+2\alpha\nu\int
_{0}^{T}\int_{\partial\Omega}u\cdot\varphi+2\nu\int_{0}^{T}\int_{\Omega}D(u):
D(\varphi)-\int_{0}^{T}\int_{\Omega}\rho uu\cdot\nabla\varphi,\nonumber\\
=\int_{0}^{T}\int_{\Omega}\rho f\cdot\varphi+\int_{\Omega}v_{0} \varphi(0),
\label{weak1}%
\end{align}
for all $\varphi\in C^{1}([0,T];H^{1}_{\sigma}(\Omega)),$ $\varphi(T,x)=0$
a.e. in $\Omega.$

\item[ii)] The initial data \eqref{prev0b}$_{3}$ is verified in the following
sense:
\[
\langle\rho(0),\psi\rangle_{W^{-1,p}(\Omega),W_{0}^{1,p^{\prime}}(\Omega
)}=\int_{\Omega}\rho_{0}\psi dx,\ \forall\ \psi\in W_{0}^{1,p^{\prime}}
(\Omega).
\]

\item[iii)] The following energy inequality
\begin{align}
\label{energyineq}\frac{1}{2}\Vert\sqrt{\rho(t)}u(t)\Vert_{2}^{2}+2\nu
\alpha\int_{0}^{t}\int_{\partial\Omega}\vert u\vert^{2}+2\nu\int_{0}^{t}\Vert
Du\Vert_{2}^{2}\leq\frac{1}{2}\Bigl\Vert\frac{v_{0}}{\sqrt{\rho_{0}}}
\Bigr\Vert_{2}^{2} +\int_{0}^{t}\int_{\Omega}\rho f\cdot u
\end{align}
holds for a.e. $t\in(0,T).$
\end{enumerate}
\end{definition}

Let us make some commentaries about the previous definition. We first notice
that the divergence-free boundary condition of the velocity field and the
boundary condition \eqref{prev0b}$_{1}$ are given by the choice of the space
$H_{\sigma}^{1}(\Omega)$. The weak formulation \eqref{weak1} also contains the
boundary condition \eqref{prev0b}$_{2}$ in the sense that if $u$ is more
regular, say $H_{\sigma}^{1}(\Omega)\cap H^{2}(\Omega)$, it can be recovered.
Indeed, first let us just recall, for the readers convenience, that the
formulation \eqref{weak1} comes from following integration by parts:

\begin{lemma}
(\cite{Iftimie-Planas})\label{part} Let $f$ and $g$ be smooth vector fields
such that $g$ is divergence free and tangent to the boundary. Then
\[
-\int_{\Omega}\Delta f\cdot g =2\int_{\Omega}D(f): D(g) - 2\int_{\partial
\Omega} [D(f)n]_{tan}\cdot g .
\]

\end{lemma}

Now, assuming that $u $ is more regular and using the previous lemma, from
\eqref{weak1}, we obtain
\[
\int_{0}^{T} \int_{\partial\Omega} [D(u)n+ \alpha u ] \cdot\varphi= 0 ,
\]
for any test function $\varphi$ satisfying $\varphi\cdot n = 0 $ on
$\partial\Omega$; consequently, $[D(u)n+ \alpha u ]_{tan} = 0$ on
$\partial\Omega$.

We also note that by taking test functions in $\mathcal{D}([0,T)\times
\overline{\Omega})$ in the form $\varphi_{h}=\psi(x)\theta_{h}(z),$ where
$\psi\in\mathcal{D}(\overline{\Omega})$ with divergente free, and $\theta
_{h}\in\mathcal{D}([0,T))$ such that $\theta_{h}(z)=1$ for $z\leq t$ and
$\theta_{h}(z)=0$ for $z\geq t+h$, and taking the limit as $h\rightarrow0,$ we
obtain an equivalent weak formulation for the momentum equation
\begin{multline*}
\int_{\Omega}(\rho u)(t)\psi+2\alpha\nu\int_{0}^{t}\int_{\partial\Omega}
u\cdot\psi+2\nu\int_{0}^{t}\int_{\Omega}D(u):D(\psi)-\int_{0}^{t}\int_{\Omega
}\rho uu\cdot\nabla\psi\\
=\int_{0}^{t}\int_{\Omega}\rho f\cdot\psi+\int_{\Omega}v_{0}\psi,
\end{multline*}
from which we deduce that $\rho u\in C_{w}([0,T];L^{\frac{2p}{p+1}}(\Omega))$.
Hence, the initial data \eqref{prev0b}$_{4}$ is verified in the following
sense: $(\rho u)(t)$ converges weakly to $v_{0}$ as $t\rightarrow0^{+}.$
\newline

The result of existence of weak solutions with finite energy is the following.

\begin{theorem}
\label{teorem1} Let $6 \leq p \leq\infty,$ $f\in L^{1}(0,T;L^{\frac{2p}{p-1}
}(\Omega)),$ $\rho_{0}\in L^{p}(\Omega),$ $\rho_{0}\geq0,$ a.e. in $\Omega,$
and $v_{0}\in L^{\frac{2p}{p+1}}(\Omega),$ $\frac{\vert v_{0}\vert^{2}}%
{\rho_{0}}\in L^{1}(\Omega).$ There exists a weak solution $ (\rho,
u ) $ of problem \eqref{prev0}-\eqref{prev0b} in the sense of
Definition \ref{definicion1}.
\end{theorem}

\begin{remark}
%Under the assumptions of Theorem \ref{teorem1},
Let us observe that the density $ \rho $ belongs to
$C([0,T];L^{p}(\Omega)).$ In fact, noting that $u\in L^{2}(0,T;H_{\sigma}%
^{1}(\Omega))\subset L^{1}(0,T;L^{6}(\Omega))$ and $1\leq p^{\prime}\leq
2\leq6,$ the desired claim follows from standard approximation and
regularization arguments for transport equations (see e.g. \cite{lionsb}).  On
the other hand, when $\rho_{0}(x)\geq c_{0}>0,$ the weak solution of
\eqref{prev0}-\eqref{prev0b} given by Theorem \ref{teorem1} also verifies
$u\in L^{\infty}(0,T;L^{2}(\Omega)).$
\end{remark}

To prove the existence of a weak solution with finite energy, we first
introduce a regularized problem, depending on a small positive parameter
$\epsilon,$ which is constructed by a regularization of the continuity and
momentum equations, as well as, a regularization of the data. More explicitly,
fixed $\epsilon>0,$ we consider the following regularized problem related to
\eqref{prev0}-\eqref{prev0b}: Find $(u_{\epsilon},\rho_{\epsilon})$ solution
of system%

\begin{equation}
\left\{
\begin{array}
[c]{rcl}%
\partial_{t}(\rho_{\epsilon}u_{\epsilon})+\mbox{div}\ (\rho_{\epsilon
}u_{\epsilon}u_{\epsilon})-\nu\Delta u_{\epsilon}+\nabla\pi_{\epsilon}
=\rho_{\epsilon}f+\frac{\epsilon}{2}(\Delta\rho_{\epsilon})u_{\epsilon
}\ \mbox{in}\ Q, &  & \\
\mbox{div}\ u_{\epsilon}=0\ \mbox{in}\ Q, &  & \\
\partial_{t}\rho_{\epsilon}+\mbox{div}\ (\rho_{\epsilon}u_{\epsilon}%
)=\epsilon\Delta\rho_{\epsilon}\ \mbox{in}\ Q, &  & \\
u_{\epsilon}\cdot n=0\ \mbox{on}\ \Sigma, &  & \\
\left[  D(u_{\epsilon})n+\alpha u_{\epsilon}\right]  _{tan}%
=0\ \mbox{on}\ \Sigma, &  & \\
\dfrac{\partial\rho_{\epsilon}}{\partial n}=0\ \mbox{on}\ \Sigma, &  & \\
\rho_{\epsilon}(0)=\rho_{0,\epsilon}\ \mbox{in}\ \Omega, &  & \\
(\rho_{\epsilon}u_{\epsilon})(0)=v_{0,\epsilon}\ \mbox{in}\ \Omega, &  &
\end{array}
\right.  \label{prev1}%
\end{equation}
where $\rho_{0,\epsilon},$  similarly to
\cite[p.149]{fir} for $6\leq p<\infty$, is such that $\rho_{0,\epsilon}\in
C^{2,r}(\bar{\Omega})$, $r\in(0,1)$ with
\begin{equation}
\left\{
\begin{array}
[c]{rcl}%
\dfrac{\partial\rho_{0,\epsilon}}{\partial n}=0\ \mbox{on}\ \partial\Omega,\  \displaystyle0<\epsilon\leq\rho_{0,\epsilon}(x),\ x\in\Omega,&  &\\
|\{x\in\Omega:\rho_{0,\epsilon }(x)<\rho_{0}(x)\}|\rightarrow0\
\mbox{as}\ \epsilon\rightarrow0, &  &\\
\rho_{0,\epsilon}\rightarrow\rho_{0}\ \mbox{in}\ L^{p}(\Omega )\
\mbox{as}\ \epsilon\rightarrow0,\ 6\leq p<\infty, &  &\\
\rho_{0,\epsilon}\rightharpoonup \rho_{0}\ \mbox{weakly-}\ast \mbox{in}\
L^{\infty}(\Omega )\  \mbox{as}\
\epsilon\rightarrow0,\  p=\infty,
\end{array}
\right.  \label{e1}%
\end{equation}
and the initial linear momentum $v_{0,\epsilon}$ is defined as
\begin{equation}
v_{0,\epsilon}(x)=\left\{
\begin{array}
[c]{rcl}%
v_{0}\ \mbox{if}\ \rho_{0,\epsilon}(x)\geq\rho_{0}(x), &  & \\
0\ \mbox{if}\ \rho_{0,\epsilon}(x)<\rho_{0}(x). &  &
\end{array}
\right.  \label{id}%
\end{equation}

We introduce the concept of weak-strong solution to the previous regularized system.

\begin{definition}
\label{defaprox} Let $6\leq p\leq\infty$. A weak-strong solution of
\eqref{prev1} is a pair of functions $(u_{\epsilon},\rho_{\epsilon})$
satisfying $u_{\epsilon}\in L^{2}(0,T;H_{\sigma}^{1}(\Omega)),\rho_{\epsilon
}\in C([0,T];W^{-1,p}(\Omega))\cap L^{\infty}(0,T;L^{p}(\Omega)),$
$\rho_{\epsilon}>0$ a.e. in $Q,$ ${\rho}_{\epsilon}{u}_{\epsilon}\in
L^{\infty}(0,T;L^{\frac{2p}{p+1}}(\Omega)),$ $\rho_{\epsilon}u_{\epsilon
}u_{\epsilon}\in L^{1}(0,T;L^{2}(\Omega)),$ $\nabla\rho_{\epsilon}\in
L^{2}(0,T;L^{2}(\Omega)),$ such that:

\begin{enumerate}
\item[i)] equation \eqref{prev1}$_{3}$ holds a.e. in $Q.$ The boundary
condition
%, the boundary condition
\eqref{prev1}$_{6}$ holds a.e. on $\Sigma,$ and the initial condition
\eqref{prev1}$_{7}$ holds a.e. in $\Omega,$

\item[ii)] the momentum equation \eqref{prev1}$_{1}$ is verified in the
following sense
\begin{multline}
-\int_{0}^{T}\int_{\Omega}\rho_{\epsilon}u_{\epsilon}\cdot\partial_{t}
\varphi+2\alpha\nu\int_{0}^{T}\int_{\partial\Omega}u_{\epsilon}\cdot
\varphi+2\nu\int_{0}^{T}\int_{\Omega}D(u_{\epsilon}) :D(\varphi)\\
-\int_{0}^{T}\int_{\Omega}\rho_{\epsilon}u_{\epsilon}u_{\epsilon}\cdot
\nabla\varphi+\frac{\epsilon}{2}\int_{0}^{T}\int_{\Omega}\nabla\rho_{\epsilon
}\nabla(u_{\epsilon}\cdot\varphi) =\int_{0}^{T}\int_{\Omega}\rho_{\epsilon
}f\cdot\varphi+\int_{\Omega}v_{0,\epsilon}\varphi(0),\label{aprox1}%
\end{multline}
for $\varphi\in C^{1}([0,T]\times\overline{\Omega}),$ with
$\mbox{div}\ \varphi=0,$ $\varphi\cdot n=0$ on $\Sigma,$ and $ \varphi(T,x) = 0 $ in $ \Omega.$
%\item[iii)] the initial condition \eqref{prev1}$_{8}$ is verified in the
%following sense
%\[
%\int_{\Omega}(\rho_{\epsilon}u_{\epsilon})(0)vdx=\int_{\Omega}v_{0,\epsilon
%}vdx,\ \ \forall v\in H_{\sigma}^{1}(\Omega)\cap W^{1,s}(\Omega
%)\ \mbox{for some fixed}\ s>3.
%\]

\end{enumerate}
\end{definition}

In order to show the existence of weak-strong solution to the regularized
problem we will consider the Galerkin approximations for the momentum equation
and then will use a limiting procedure. The existence of solutions for this
approximate problem will be obtained by linearization and the Schauder fixed
point theorem. We can now state the result of existence weak-strong solution
for the regularized system \eqref{prev1}.

\begin{proposition}
\label{prop1} Let $\rho_{0,\epsilon}$ be as in \eqref{e1} and $p,u_{0},f$ as
in Theorem \ref{teorem1}. Then, there exists $(u_{\epsilon},\rho_{\epsilon})$
a weak-strong solution of \eqref{prev1}, in the sense of Definition
\ref{defaprox}. Moreover, $(u_{\epsilon},\rho_{\epsilon})$ verifies  $\rho_{\epsilon}u_{\epsilon}\in
L^{2}(0,T;L^{\frac{6p}{p+6}}(\Omega))$, $\rho_{\epsilon}\in L^{\varrho
}(0,T;W^{2,\varrho}(\Omega))$ and $\partial_{t}\rho_{\epsilon}\in L^{\varrho
}(0,T;L^{\varrho}(\Omega))$, for some $\varrho\geq\frac{3}{2}.$
\end{proposition}

%\begin{proposition}
%\label{prop1} Let $\rho_{0,\epsilon}$ be as in \eqref{e1} and
%$p,u_{0},f$ as in Theorem \ref{teorem1}. Then, there exists
%$(u_{\epsilon},\rho_{\epsilon}) $ a weak-strong solution of
%\eqref{prev1}, in the sense of Definition \ref{defaprox}. Moreover,
%$(u_\epsilon,\rho_\epsilon)$ verifies that $\rho_\epsilon\in
%L^2(0,T;L^2(\Omega))\cap C([0,T];W^{-1,p}(\Omega)),$ $\rho_\epsilon
%u_\epsilon\in L^2(0,T;L^{\frac{6p}{p+6}}(\Omega))\cap
%L^\infty(0,T;L^{\frac{2p}{p+1}}(\Omega)),$
%$\sqrt{\rho_\epsilon}u_\epsilon\in L^\infty(0,T;L^2(\Omega)),$
%$\rho_\epsilon u_\epsilon u_\epsilon\in L^1(0,T;L^2(\Omega)).$
%\end{proposition}

\begin{proof}
We split the proof into five steps. \newline

\textbf{Step 1: Approximate problem.} Let $\{w^{k}\}_{k\in\mathbb{N}}$ be a
smooth basis of $H^{1}_{\sigma}(\Omega)$, orthonormal in $L^{2}(\Omega)$ and
let $\mathcal{Y}^{m} = span\{w^{1},...,w^{m}\}.$
%such that $(w^{i},w^{j})=\delta_{ij},$ for all
%$i,j\in\mathbb{N}.$ We denote by $\mathcal{Y}^{m}=span\{w^{1},...,w^{m}\}.$
Let us consider a sequence $\{f^{m}\}_{m \in\mathbb{N}}$ in $C ([0,T];L^{\frac
{2p}{p-1}}(\Omega))$ such that $f^{m}\rightarrow f$ in $L^{1}(0,T;L^{\frac
{2p}{p-1}}(\Omega)).$
%(Here if $p=2$ then $L^{\frac {2p}{p-2}}(\Omega)$ is
%$L^{\infty}(\Omega)$).

For each $m\in\mathbb{N}$ consider the problem of finding $\rho^{m}\in
C([0,T];C^{2}(\bar{\Omega}))$ and $u^{m}\in C^{1}([0,T];\mathcal{Y}^{m})$
satisfying
\begin{equation}
\left\{
\begin{array}
[c]{rcl}%
\partial_{t}\rho^{m}+\mbox{div}\ (\rho^{m}u^{m})=\epsilon\Delta\rho
^{m}\ \mbox{in}\ Q, &  & \\
\dfrac{\partial\rho^{m}}{\partial n}=0\ \mbox{on}\ \Sigma, &  & \\
\rho^{m}(0)=\rho_{0,\epsilon}\ \mbox{in}\ \Omega, &  &
\end{array}
\right.  \label{prev2}%
\end{equation}%
\begin{equation}
\left\{
\begin{array}
[c]{rcl}%
\displaystyle\int_{\Omega}\left\{  \partial_{t}(\rho^{m}u^{m})\cdot v-\rho
^{m}u^{m}u^{m}\cdot\nabla v-\rho^{m}f^{m}\cdot v+2\nu Du^{m}:Dv\right\}  dx &
& \\
\displaystyle+2\alpha\nu\int_{\partial\Omega}u^{m}\cdot v=-\frac{\epsilon}%
{2}\int_{\Omega}\nabla\rho^{m}\nabla(u^{m}\cdot v),\ \forall v\in
\mathcal{Y}^{m}, &  & \\
u^{m}(0)=u_{0}^{m}, &  &
\end{array}
\right.  \label{prev3}%
\end{equation}
where $u_{0}^{m}\in\mathcal{Y}^{m}$ is uniquely determined by
\begin{equation}
\int_{\Omega}\rho_{0,\epsilon}u_{0}^{m}\phi dx=\int_{\Omega}v_{0,\epsilon}\phi
dx\ \mbox{for all}\ \phi\in\mathcal{Y}^{m}. \label{din}%
\end{equation}

For the sake of simplicity, we shall omit the dependence of $(u^{m},\rho^{m})$
on the parameter $\epsilon.$ Observe that $u_{0}^{m}$ is well-defined because
the matrix with coefficients $\int_{\Omega}\rho_{0,\epsilon}w^{j}w^{i}dx$ is
invertible. Moreover, by using (\ref{din}), the definition of $v_{0,\epsilon}$
(see (\ref{id})) and the assumption $\frac{|v_{0}|^{2}}{\rho_{0}}\in
L^{1}(\Omega),$ the following estimate holds true:
\begin{equation}
\int_{\Omega}v_{0,\epsilon}u_{0}^{m}dx=\int_{\Omega}\rho_{0,\epsilon}%
|u_{0}^{m}|^{2}dx\leq\int_{\Omega}\frac{|v_{0,\epsilon}|^{2}}{\rho
_{0,\epsilon}}dx\leq\int_{\Omega}\frac{|v_{0}|^{2}}{\rho_{0}}dx\leq C.
\label{estdato}%
\end{equation}
\newline

\textbf{Step 2: Existence of solutions to the approximate problem.} Fixed
$m\in\mathbb{N},$ the existence of approximate solutions $(u^{m},\rho^{m})$ of
(\ref{prev2})-(\ref{prev3}) is proved by linearization and the Schauder fixed
point theorem. In fact, fixed $m\in\mathbb{N}$ and given $w\in
C([0,T];\mathcal{Y}^{m}),$ the following lemma gives the existence of
$\rho^{m}\in C([0,T];C^{2}(\bar{\Omega}))$ such that
\begin{equation}
\left\{
\begin{array}
[c]{rcl}%
\partial_{t}\rho^{m}+\mbox{div}\ (\rho^{m}w)=\epsilon\Delta\rho^{m}%
\ \mbox{in}\ Q, &  & \\
\dfrac{\partial\rho^{m}}{\partial n}=0\ \mbox{on}\ \Sigma, &  & \\
\rho^{m}(0)=\rho_{0,\epsilon}\ \mbox{in}\ \Omega. &  &
\end{array}
\right.  \label{prev4}%
\end{equation}
%\begin{lemma}
%\label{l1}(\cite{fir}) Let $T>0$ be an arbitrary real number. Suppose that the
%initial condition $0<\rho_{0,\epsilon}\in C^{2+r}(\bar{\Omega})$ and satisfies
%the compatibility condition $\frac{\partial\rho_{0,\epsilon} }{\partial n}=0$
%on $\partial\Omega.$ Then, for a given $w\in C([0,T];\mathcal{Y}^{m})\subset
%C([0,T];C_{0}^{2}(\bar{\Omega})),$ Problem \eqref{prev4} possesses a unique
%classical solution $\rho^{m}\in C([0,T];C_{0}^{2+r}(\bar{\Omega})),$
%$\partial_{t}\rho^{m}\in C([0,T];C_{0}^{r}(\bar{\Omega})).$ Moreover, for all
%$t\in\lbrack0,T],\ x\in\Omega$, it holds:
%\begin{align}
%\inf_{x\in\Omega}\rho_{0,\epsilon}(x)\exp\left(  -\int_{0}^{t}\Vert
%\mbox{div}\ w(s)\Vert_{\infty}ds \right)  \leq\rho^{m}(x,t)\nonumber\\
%\leq\sup_{x\in\Omega}\rho_{0,\epsilon}(x)\exp\left(  -\int_{0}^{t}%
%\Vert\mbox{div}\ w(s)\Vert_{\infty}ds \right)  . \label{aux2}%
%\end{align}
%\end{lemma}
%\textcolor{blue}{Me parece que el próximo lema está mas claro, con $\mathcal{Y}^{m} \subset H^1_\sigma$. Habria que cambiar la regularidad
%$\rho_{0,\epsilon}$.}

\begin{lemma}
(\cite[Lemma 3.1]{Feireisl}) \label{lemarho} Let $w\in C([0,T];\mathcal{Y}%
^{m})$ be a given vector field. Suppose that $\rho_{0,\epsilon}\in
C^{2,r}(\bar{\Omega}),$ $r\in(0,1),$ $\inf_{x\in\Omega}\rho_{0,\epsilon}(x)>0$
and satisfies the compatibility condition $\frac{\partial\rho_{0,\epsilon}%
}{\partial n}=0$ on $\partial\Omega.$ Then problem \eqref{prev4} possesses a
unique classical solution
\[
\rho^{m}=\rho^{m}(w)\in\mathcal{W}=\{\rho^{m}\in C([0,T];C^{2,r}(\bar{\Omega
})),\ \partial_{t}\rho^{m}\in C([0,T];C^{0,r}(\bar{\Omega}))\}.
\]
Moreover, the mapping $w\rightarrow\rho^{m}(w)$ maps bounded sets in
$C([0,T];\mathcal{Y}^{m})$ into bounded sets in $\mathcal{W}$ and it is
continuous with values in $C^{1}([0,T]\times\bar{\Omega})$.

Finally, as $div(w)=0$, it holds
\begin{equation}
\inf_{x\in\Omega}\rho_{0,\epsilon}(x)\leq\rho^{m}(x,t)\leq\sup_{x\in\Omega
}\rho_{0,\epsilon}(x),\ \ t\in\lbrack0,T],\ x\in\Omega. \label{aux-pho}%
\end{equation}

\end{lemma}

%\begin{remark}
%\label{Linear1} The previous lemma also holds true in the case
%$T=\infty,$ by replacing $C$ and $[0,T]$ by $BC$ and $[0,\infty),$
%respectively. Notice that the inequalities in (\ref{aux2}) are
%independent of $T.$
%\end{remark}
%\bigskip

Posteriorly, given $w\in C([0,T];\mathcal{Y}^{m})$ and $\rho^{m}\in
C([0,T];C^{2}(\bar{\Omega})),$ we solve the following linear problem:

Find $u^{m}(x,t)=\sum_{i=1}^{m}\psi_{i}(t)w^{i}(x)$ satisfying
\[
\left\{
\begin{array}
[c]{rcl}%
\displaystyle\int_{\Omega}\left\{  \partial_{t}(\rho^{m}u^{m})\cdot v-\rho
^{m}u^{m}w\cdot\nabla v-\rho^{m}f^{m}\cdot v+2\nu Du^{m}:Dv\right\}  dx &  &
\\
\displaystyle+2\alpha\nu\int_{\partial\Omega}u^{m}\cdot v+\frac{\epsilon}
{2}\int_{\Omega}\nabla\rho^{m}\nabla(u^{m}\cdot v) = 0,\ \forall
v\in\mathcal{Y}^{m}, &  & \\
u^{m}(0)=u_{0}^{m}, &  &
\end{array}
\right.
\]
which is equivalent to
\begin{equation}
\left\{
\begin{array}
[c]{rcl}%
\displaystyle\int_{\Omega}\left\{  \rho^{m}(\partial_{t}u^{m}+(w\cdot
\nabla)u^{m}-f^{m})\cdot v+2\nu Du^{m}:Dv\right\}  dx &  & \\
+2\alpha\nu\displaystyle\int_{\partial\Omega}u^{m}\cdot v+\frac{\epsilon}
{2}\displaystyle\int_{\Omega}\nabla\rho^{m}\nabla(u^{m}\cdot v)=0,\ \forall
v\in\mathcal{Y}^{m}, &  & \\
u^{m}(0)=u_{0}^{m}. &  &
\end{array}
\right.  \label{prev5}%
\end{equation}

Notice that \eqref{prev5} leads us to a linear system of ordinary differential
equations for $\{\psi_{j}(t)\}_{j=1}^{m}$:
\begin{equation}
\left\{
\begin{array}
[c]{rcl}%
\displaystyle\sum_{j=1}^{m}a_{ij}^{1}(t)\dfrac{d\psi_{j}}{dt}+\sum_{j=1}%
^{m}a_{ij}^{2}(t)\psi_{j}+a_{i}^{3}(t)=0,\ \mbox{in}\ (0,T),\ 1\leq i\leq m, &
& \\
\{\psi_{j}(0)\}_{j=1}^{m}\equiv\mbox{components of}\ u_{0}^{m}, &  &
\end{array}
\right.  \label{prev6}%
\end{equation}
where
\begin{align*}
a_{ij}^{1}  &  =\int_{\Omega}\rho^{m}w^{j}w^{i}dx\in C^{1}([0,T]),\\
a_{ij}^{2}  &  =\int_{\Omega}\{(\rho^{m}w\cdot\nabla w^{j})\cdot w^{i}+2\nu
Dw^{j}:Dw^{i}\}dx+2\alpha\nu\int_{\partial\Omega}w^{j}\cdot w^{i}\\
&  +\frac{\epsilon}{2}\int_{\Omega}\nabla\rho^{m}\nabla(w^{j}\cdot w^{i})dx\in
C^{1}([0,T]),\\
a_{i}^{3}  &  =-\int_{\Omega}\rho^{m}f^{m}w^{i}dx\in C^{1}([0,T]).
\end{align*}
Since $\rho^{m}\geq\epsilon,$ it holds
\[
\sum_{i,j}a_{ij}^{1}(t)\xi_{i}\xi_{j}=\int_{\Omega}\rho^{m}(x,t)(\sum
_{i=1}^{m}\xi_{i}w^{i}(x))^{2}dx\geq\epsilon\sum_{i=1}^{m}|\xi_{i}%
|^{2},\ \forall\xi\in\mathbb{R}^{m}.
\]
It follows that the matrix $A=\{a_{ij}^{1}\}_{i,j}$ is symmetric and positive
definite, and in particular $A$ is invertible. From the classical theory of
ordinary differential equations, system \eqref{prev6} has a unique solution
$\{\psi_{j}\}_{j=1}^{m}\in(C^{1}([0,T]))^{m}$; then, the solvability of the
system \eqref{prev5} is guaranteed.

Multiplying equation \eqref{prev4}$_{1}$ by $\frac{1}{2}|u^{m}|^{2},$
integrating on $\Omega$ and adding the result to \eqref{prev5} with
$v=u^{m}(t),$ after some calculations, we obtain%

\[
\frac{1}{2}\frac{d}{dt}\Vert\sqrt{\rho^{m}}u^{m}(t)\Vert_{2}^{2}+2\nu\Vert
Du^{m}(t)\Vert_{2}^{2}+2\nu\alpha\int_{\partial\Omega}|u^{m}|^{2}=\int
_{\Omega}\rho^{m}f^{m}\cdot u^{m}.
\]
Thus, the H\"{o}lder inequality implies
\[
\frac{1}{2}\frac{d}{dt}\Vert\sqrt{\rho^{m}}u^{m}(t)\Vert_{2}^{2}\leq\Vert
\sqrt{\rho^{m}}u^{m}(t)\Vert_{2}\Vert\sqrt{\rho^{m}}f^{m}(t)\Vert_{2}.
\]
Using a generalized Gronwall lemma (\cite[Lemma 5]{Simon}), from last
inequality we can obtain
\[
\Vert\sqrt{\rho^{m}}u^{m}(t)\Vert_{2}\leq\Vert\sqrt{\rho_{0,\epsilon}}%
u_{0}^{m}\Vert_{2}+\int_{0}^{T}\Vert\sqrt{\rho^{m}}f^{m}\Vert_{2}.
\]
From the previous inequality, (\ref{aux-pho}) and (\ref{estdato}), we conclude that $u^{m}$ is bounded in
$C([0,T];L^{2}(\Omega))$ (independent of $m$ and $w$). Thus, as $u^{m}
(x,t)=\sum_{j=1}^{m}\psi_{j}(t)w^{j}(x)$ and $\sum_{j=1}^{m}|\psi_{j}
(t)|^{2}=\Vert u^{m}\Vert_{2}^{2},$ then $\{\psi_{j}\}$ is bounded in
$C([0,T])$ which implies that $u^{m}$ is bounded (independent of $w$) in
$C([0,T];\mathcal{Y}^{m}).$ Moreover, if $w$ is bounded in
$C([0,T];\mathcal{Y}^{m}),$ from (\ref{prev6}) and the symmetry of
$A=\{a_{ij}^{1}\},$ the set $\{\partial_{t}\psi_{j}:1\leq j\leq m\}$ is
bounded in $C([0,T])$ which implies that $\partial_{t}u^{m}$ is bounded in
$C([0,T];\mathcal{Y}^{m}).$ Thus, we conclude that $u^{m}$ is bounded in
$C^{1}([0,T];\mathcal{Y}^{m})$ provided $w$ is bounded in $C([0,T];\mathcal{Y}
^{m}).$

Given $w\in C([0,T];\mathcal{Y}^{m}),$ from Lemma \ref{lemarho} there exists a
unique $\rho^{m}\in C([0,T];C^{2,r}(\bar{\Omega}))$ solution of (\ref{prev4}).
Knowing $w$ and $\rho^{m}$ there exists a unique solution $u^{m}$ of
(\ref{prev5}) in $C^{1}([0,T];\mathcal{Y}^{m})$ which is bounded (independent
of $m$) in $C^{1}([0,T];\mathcal{Y}^{m}),$ provided $w$ is bounded in
$C([0,T];\mathcal{Y}^{m}).$ Thus, there are $M_{1},M_{2}>0$ such that $\Vert
u^{m}\Vert_{C^{1}([0,T];\mathcal{Y}^{m})}\leq M_{2}$ if $\Vert w\Vert
_{C([0,T];\mathcal{Y}^{m})}\leq M_{1}.$ Denote by $B_{1}$ the closed ball in
$C([0,T];\mathcal{Y}^{m})$ of radio $M_{1},$ and $B_{2}$ the closed ball in
$C^{1}([0,T];\mathcal{Y}^{m})$ of radio $M_{2}.$ Then, the mapping
\begin{align*}
\mathcal{T}  &  :B_{1}\rightarrow B_{2}\\
&  \ \ w\mapsto u^{m}%
\end{align*}
is continuous. The Arzel\`{a}-Ascoli theorem implies that $B_{2}\subset
C([0,T];\mathcal{Y}^{m})$ compactly and therefore the mapping $\mathcal{T}$ is
continuous and compact from $B_{1}$ into $B_{1}.$ Then, the Schauder fixed
point theorem implies the existence of a fixed point $u^{m}$ for a given
$T>0.$ Taking $\rho^{m}$ the corresponding solution of (\ref{prev4}), we
obtain the existence of an approximate solution $(u^{m},\rho^{m})$ of
(\ref{prev2})-(\ref{prev3}). \newline

\textbf{Step 3: Estimates for $(u^{m},\rho^{m})$.} We will obtain several
estimates for the approximate solution $(u^{m},\rho^{m})$ which are
independent of $m$ and, in general, are also independent of $\epsilon>0.$ In
the sequel, $C$ will denote a constant independent of $m$ and $\epsilon$ that
may change from an estimate to another. If $C$ depends on $\epsilon$ we shall
indicate this dependence by $C(\epsilon).$

The first estimate comes from Lemma \ref{lemarho},
\[
0<\inf_{x\in\Omega}\rho_{0,\epsilon}(x)\leq\rho^{m}(x,t).
\]
Next, we multiply \eqref{prev2}$_{1}$ by $\rho^{m}|\rho^{m}|^{p-2}$, where
$p>2$, integrate by parts and use the incompressibility of the flow $u^{m}$ to
obtain
\[
\frac{1}{p}\frac{d}{dt}\Vert\rho^{m}(t)\Vert_{p}^{p}+\epsilon(p-1)\int
_{\Omega}||\nabla\rho^{m}||\rho^{m}|^{\frac{p-2}{2}}|^{2}=0,
\]
then
\begin{equation}
\ \Vert\rho^{m}\Vert_{L^{\infty}(0,T;L^{p}(\Omega))}\leq\Vert\rho_{0,\epsilon
}\Vert_{L^{p}(\Omega)}\leq C,\label{aux-est-rho1}%
\end{equation}
because \eqref{e1}$_{3}$ and \eqref{e1}$_{4}$. Also, we have that
\begin{equation} \label{energyrho}
\Vert\rho^{m}(t)\Vert_{2}^{2}+2\epsilon \int_0^t \Vert\nabla
\rho^{m}(s)\Vert_{2}^{2} ds= \Vert \rho_{0,\epsilon}\Vert_2^2,
\end{equation}
so, it follows that
\begin{equation}
(\sqrt{\epsilon}\nabla\rho^{m},\sqrt{\epsilon}\Delta\rho^{m}%
)\ \mbox{is uniformly bounded in}\ L^{2}(0,T;L^{2}(\Omega)\times
W^{-1,2}(\Omega)). \label{aux-j}%
\end{equation}

%Moreover, $\rho^{m}$ satisfies the classical estimates (see \cite[Th.10.22]%
%{Feireisl})
%\begin{equation}
%\Vert\partial_{t}\rho^{m}\Vert_{L^{p}(0,T;L^{p}(\Omega))}+\Vert\rho^{m}%
%\Vert_{L^{p}(0,T;W^{2,{p}}(\Omega))}\leq C(\epsilon). \label{aux-reg-rho1}%
%\end{equation}
%Moreover
%\begin{equation}
%{\sqrt{\epsilon}}\Delta\rho^{m}\ \mbox{is bounded in}\
%L^{2}((0,\infty);W^{-1,2}(\Omega)).
%\label{j31}
%\end{equation}
Integrating equation \eqref{prev2}$_{1}$ with respect to the space
variable and using that $\frac{\partial\rho^{m}}{\partial n}|_{\Sigma}=0,$
yield the total mass conservation
\begin{equation}
\int_{\Omega}\rho^{m}(t)dx=\int_{\Omega}\rho_{0,\epsilon}dx,\ t\in\lbrack0,T].
\label{normaL1}%
\end{equation}
%Multiplying (\ref{prev2}) by $\frac{1}{2}u^{m}\cdot v,$ $v\in
%\mathcal{Y}^{m},$ integrating on $\Omega$ and adding the result with
%(\ref{prev3}) we obtain:
%\begin{eqnarray*}
%\displaystyle\int_{\Omega}\left\{  \rho^{m}(\frac{\partial
%u^{m}}{\partial t}+{u^{m}}\frac{\partial\rho^{m}}{\partial
%t}+(\rho^{m} u^m\cdot\nabla) u^{m}+{u^{m}}\mbox{div}(\rho^{m}
%u^{m}) -\rho^m f^{m})\cdot v+2\nu D u^{m}:D v\right\}  dx\\
%+\int_\Omega\frac{\partial\rho^m}{\partial t}\frac{u^m\cdot
%v}{2}+\int_\Omega\mbox{div}(\rho^mu^m)\frac{u^m\cdot
%v}{2}+2\alpha\nu\int_{\partial\Omega}u^m\cdot
%v={\epsilon}\int_\Omega\Delta\rho^m(u^m\cdot v).
%\end{eqnarray*}

Notice that equation \eqref{prev3} is equivalently to
\begin{align}
&  \displaystyle\int_{\Omega}\left\{  \rho^{m}(\partial_{t}u^{m}+(u^{m}%
\cdot\nabla)u^{m}-f^{m})\cdot v+2\nu Du^{m}:Dv\right\}  dx\nonumber\\
&  +2\alpha\nu\int_{\partial\Omega}u^{m}\cdot
v-\frac{\epsilon}{2}\int _{\Omega}\nabla\rho^{m}\nabla(u^{m}\cdot v)
={0}\ \forall v\in\mathcal{Y}^{m}.
\label{g1}%
\end{align}
Multiplying equation \eqref{prev2}$_{1}$ by $\frac{1}{2}|u^{m}|^{2},$
integrating on $\Omega$ and adding the result with \eqref{g1} for
$v=u^{m}(t),$ we find after some calculations
%\begin{align*}
%\displaystyle\int_{\Omega}\left\{\partial_t(\rho^{m}
%\frac{\vert u^{m}\vert^{2}}{2}) % +\mbox{div} (\rho^{m} u^m\frac{\vert u^{m}\vert^{2}}{2})
%+2\nu\vert D u^{m}\vert^{2}-\rho^mf^{m}\cdot
%u^{m}\right\}  dx+2\nu\alpha\int_{\partial\Omega}\vert u^m\vert^2=0,
%\end{align*}
%which implies%
\begin{equation}
\frac{1}{2}\frac{d}{dt}\Vert\sqrt{\rho^{m}}u^{m}(t)\Vert_{2}^{2}+2\nu\Vert
Du^{m}(t)\Vert_{2}^{2}+2\nu\alpha\int_{\partial\Omega}|u^{m}|^{2}=\int
_{\Omega}\rho^{m}f^{m}\cdot u^{m}. \label{e2aaaa}%
\end{equation}
Thus, the H\"{o}lder inequality implies
\[
\frac{1}{2}\frac{d}{dt}\Vert\sqrt{\rho^{m}}u^{m}(t)\Vert_{2}^{2}\leq\Vert
\sqrt{\rho^{m}}u^{m}(t)\Vert_{2}\Vert\sqrt{\rho^{m}}f^{m}(t)\Vert_{2}.
\]
Applying a generalized Gronwall lemma (\cite[Lemma 5]{Simon}), it follows
that
\[
\Vert\sqrt{\rho^{m}}u^{m}(t)\Vert_{2}\leq\Vert\sqrt{\rho_{0,\epsilon}}%
u_{0}^{m}\Vert_{2}+\int_{0}^{T}\Vert\sqrt{\rho^{m}}f^{m}\Vert_{2},
\]
thenceforth, by applying again H\"{o}lder's inequality, and using
\eqref{estdato} and \eqref{aux-est-rho1}, we obtain
\[
\Vert\sqrt{\rho^{m}}u^{m}(t)\Vert_{2}\leq\Vert\sqrt{\rho_{0,\epsilon}}%
u_{0}\Vert_{2}+\Vert\rho^{m}\Vert_{L^{\infty}(0,T;L^{p}(\Omega))}^{1/2}\Vert
f^{m}\Vert_{L^{1}(0,T;L^{\frac{2p}{p-1}}(\Omega))}\leq C,
\]
because $(f^{m})$ is bounded in $L^{1}(0,T;L^{\frac{2p}{p-1}}(\Omega)).$

On the other hand, by integrating \eqref{e2aaaa} on $(0,t)$ and proceeding
similarly as before, we arrive at
\begin{align*}
\frac{1}{2}  &  \Vert\sqrt{\rho^{m}}u^{m}(t)\Vert_{2}^{2}+2\nu\int_{0}%
^{t}\Vert Du^{m}\Vert_{2}^{2}+2\nu\alpha\int_{0}^{t}\int_{\partial\Omega
}|u^{m}|^{2} \\% \label{e5}\\
&  \leq\frac{1}{2}\Vert\sqrt{\rho_{0,\epsilon}}u_{0}^{m}\Vert_{2}^{2}+\int
_{0}^{t}\Vert\sqrt{\rho^{m}}u^{m}\Vert_{2}\Vert\sqrt{\rho^{m}}f^{m}\Vert
_{2}\nonumber\\
&  \leq\frac{1}{2}\Vert\sqrt{\rho_{0,\epsilon}}u_{0}\Vert_{2}^{2}+C\Vert
\sqrt{\rho^{m}}u^{m}\Vert_{L^{\infty}(0,T;L^{2}(\Omega))}\Vert\rho^{m}%
\Vert_{L^{\infty}(0,T;L^{p}(\Omega))}^{1/2}\Vert f^{m}\Vert_{L^{1}%
(0,T;L^{\frac{2p}{p-1}}(\Omega))}.\nonumber
\end{align*}
Therefore, $(Du^{m})$ is bounded in $L^{2}(0,T;L^{2}(\Omega)).$ To
estimate $(u_{m})$ in $L^{2}(0,T;H^{1}(\Omega))$ we shall apply the following
generalized Korn inequality (see \cite[Th. 10.17]{Feireisl})
\begin{align}
\Vert v\Vert_{H^{1}(\Omega)}^{2}  &  \leq C(K,M,p)\left(  \Vert D(v)\Vert_{2}
^{2}+\Vert Rv\Vert_{1}^{2}\right) \nonumber\\
&  \leq C(K,M,p)\left(  \Vert D(v)\Vert_{2} ^{2}+ \| R \|_{1} \Vert
Rv^{2}\Vert_{1}\right)  ,\label{Korn}%
\end{align}
for $v\in H_{\sigma}^{1}(\Omega)$ and any function $R\geq0$ such that
$0<M\leq\int_{\Omega}Rdx$, $\Vert R\Vert_{p}\leq K,$ for some $p>1$.
Without loss of generality we can assume that $\rho_{0}\geq0$ and $\rho
_{0}\not \equiv 0$. It follows from \eqref{normaL1} and $\rho_{0,\epsilon
}\rightarrow\rho_{0}$ in $L^{p}(\Omega),$ $6\leq p<\infty,$ that%
\[
M=\frac{1}{2}\left\Vert \rho_{0}\right\Vert _{1}\leq\int_{\Omega}\rho
^{m}dx=\int_{\Omega}\rho_{0,\epsilon}dx\leq2\left\Vert \rho_{0}\right\Vert
_{1},\text{ for small }\epsilon>0,
\]
so we can take $R=\rho^{m}$ in (\ref{Korn}) and use previous estimates in
order to infer that $(u_{m})$ is bounded in $L^{2}(0,T;H^{1}(\Omega))\subset
L^{2}(0,T;L^{6}(\Omega)),$ independently of $m$ and $\epsilon$. In the case
$p=\infty,$ we also can choose the same $M$ since $\rho_{0,\epsilon},\rho
_{0}\geq0$ and the weak-$\ast$ convergence $\rho_{0,\epsilon}\rightharpoonup
\rho_{0}$ in $L^{\infty}$ implies that $\int_{\Omega}\rho_{0,\epsilon
}dx\rightarrow\int_{\Omega}\rho_0 dx$.

Next, by applying H\"{o}lder's inequality, we estimate
\[
\Vert\sqrt{\rho^{m}}u^{m}\Vert_{L^{2}(0,T;L^{\frac{12p}{6+2p}}(\Omega))}
\leq\Vert\rho^{m}\Vert_{L^{\infty}(0,T;L^{p}(\Omega))}^{1/2}\Vert u^{m}
\Vert_{L^{2}(0,T;L^{6}(\Omega))}\leq C.
\]
Since $p\geq6$, we have that $\frac{12p}{6+2p}\geq4$, hence we can use
interpolation to obtain
\[
\Vert\sqrt{\rho^{m}}u^{m}\Vert_{4}\leq\Vert\sqrt{\rho^{m}}u^{m}\Vert
_{2}^{(p-6)/(4p-6)}\Vert\sqrt{\rho^{m}}u^{m}\Vert_{\frac{12p}{6+2p}
}^{3p/(4p-6)}.
\]
Thus, taking $\varsigma=\frac{2(4p-6)}{3p}$, it follows that $\varsigma\geq2$
and
\[
\Vert\sqrt{\rho^{m}}u^{m}\Vert_{4}^{\varsigma}\leq\Vert\sqrt{\rho^{m}}
u^{m}\Vert_{2}^{2(p-6)/3p}\Vert\sqrt{\rho^{m}}u^{m}\Vert_{\frac{12p}{6+2p}
}^{2}.
\]
Consequently, $(\sqrt{\rho^{m}}u^{m})$ is bounded in $L^{\varsigma}
(0,T;L^{4}(\Omega)),$ which implies that
\begin{equation}
\Vert\rho^{m}u^{m}u^{m}\Vert_{L^{\varsigma/2}(0,T;L^{2}(\Omega))}\leq C.
\label{e10}
\end{equation}

By using that $(\rho^{m}) $, $(\sqrt{\rho^{m}}u^{m})$ and $(u^{m}) $ are
bounded in $L^{\infty}(0,T;L^{p}(\Omega)) $, $L^{\infty}(0,T;L^{2}(\Omega))$
and $L^{2}(0,T;L^{6}(\Omega))$ respectively, we deduce the following bounds
for $(\rho^{m}u^{m}) $,
\begin{align}
\Vert\rho^{m}u^{m}\Vert_{L^{\infty}(0,T;L^{\frac{2p}{p+1}} (\Omega))}  &
\leq\|\rho^{m}\|_{L^{\infty}(0,T;L^{p}(\Omega))}^{\frac{1}{2}} \|\sqrt
{\rho^{m}}u^{m} \|_{L^{\infty}(0,T;L^{2}(\Omega))} \leq C,\label{j21a}\\
\Vert\rho^{m}u^{m}\Vert_{ L^{2}(0,T;L^{\frac{6p}{p+6}}(\Omega))}  &  \leq\|
\rho^{m} \|_{L^{\infty}(0,T;L^{p}(\Omega))} \|u^{m}\|_{L^{2}(0,T;L^{6}
(\Omega))}\leq C. \label{j21}
\end{align}
By interpolation it follows that
\[ \Vert\rho^{m}u^{m}\Vert_2 \leq \Vert\rho^{m}u^{m}\Vert_{L^{\infty}(0,T;L^{\frac{2p}{p+1}} (\Omega))}^\theta \Vert\rho^{m}u^{m}\Vert_{\frac{6p}{p+6}}^{1-\theta},
\]
where $ 0 < \theta = \frac{2p-6}{2p-3} < 1.$ Thus,
\[ \Vert\rho^{m}u^{m}\Vert_{L^q(0,T;L^2(\Omega))}  \leq C , \ \text{ with } q = \frac{2(2p-3)}{3} \geq 6. \]
By using the maximal regularity for  parabolic equations, we obtain
that $ (\rho^m) $ is bounded in $L^q(0,T;H^{1}(\Omega))$, which
implies that $ (\nabla \rho^m) $ is bounded in
$L^q(0,T;L^2(\Omega))$. This bound together the estimate for $u^m $
in $L^2(0,T;L^6(\Omega))$ allow to apply the classical $L^\varrho-$
theory of parabolic equations  (see \cite[Th.10.22]{Feireisl})  to
conclude that
\begin{equation}
\Vert\partial_{t}\rho^{m}\Vert_{L^{\varrho}(0,T;L^{\varrho}(\Omega))}+\Vert\rho^{m}
\Vert_{L^{\varrho}(0,T;W^{2,{\varrho}}(\Omega))}\leq C(\epsilon),
\text{ with } \varrho = \frac{2q}{q+2} \geq \frac{3}{2} . \label{aux-reg-rho1}
\end{equation}

Finally, we are going to estimate the derivative in time of $\rho^{m}$ and
$\rho^{m}u^{m}.$ To this end, consider the equation $\partial_{t}\rho
^{m}=-\mbox{div}\ (\rho^{m}u^{m})+\epsilon\Delta\rho^{m}$ in $Q,$ and notice
that $W^{-1,2}(\Omega)\subset W^{-1,\frac{2p}{p+1}}(\Omega),$ then, by using
\eqref{aux-j} and \eqref{j21a}, we conclude that
\begin{equation}
\Vert\partial_{t}\rho^{m}\Vert_{L^{2}(0,T;W^{-1,\frac{2p}{p+1}}(\Omega))}\leq
C(\Vert\rho^{m}u^{m}\Vert_{L^{\infty}(0,T;L^{\frac{2p}{p+1}}(\Omega))}%
+\sqrt{\epsilon}\Vert\sqrt{\epsilon}\Delta\rho^{m}\Vert_{L^{2} (0,T;W^{-1,2}%
(\Omega))}) \leq C, \label{aux-deriv-time}%
\end{equation}
for small $\epsilon>0$.

From the momentum equation in \eqref{prev3} together with \eqref{e10}, the
H\"{o}lder inequality and Sobolev imbedding, for all $v\in\mathcal{D}(\Omega)$
and $s>3,$ we have
\begin{align*}
&  \Bigl|\frac{d}{dt}\int_{\Omega}\rho^{m}u^{m}v\Bigr|\leq\Vert\rho^{m}%
u^{m}u^{m}\Vert_{2}\Vert\nabla v\Vert_{2}+2\nu\Vert Du^{m}\Vert_{2}\Vert
Dv\Vert_{2}+\Vert\rho^{m}\Vert_{p}\Vert f^{m}\Vert_{\frac{2p}{p-1}}\Vert
v\Vert_{\frac{2p}{p-1}}\nonumber\\
&  \hspace{0.5cm}+2\alpha\nu\Vert u^{m}\Vert_{L^{2}(\partial\Omega)}\Vert
v\Vert_{L^{2}(\partial\Omega)}+\frac{\epsilon}{2}\Vert\nabla\rho^{m}\Vert
_{2}\Vert u^{m}\Vert_{6}\Vert\nabla v\Vert_{3}+\frac{\epsilon}{2}\Vert
\nabla\rho^{m}\Vert_{2}\Vert\nabla u^{m}\Vert_{2}\Vert v\Vert_{\infty
}\nonumber\\
&  \hspace{0.5cm}\leq C\bigl(\Vert\rho^{m}u^{m}u^{m}\Vert_{2}+\Vert\rho
^{m}\Vert_{p}\Vert f^{m}\Vert_{\frac{2p}{p-1}}+\Vert u^{m}\Vert_{H^{1}%
(\Omega)}\nonumber\\
&  \hspace{0.5cm}+\epsilon\Vert\nabla\rho^{m}\Vert_{2}\Vert u^{m}\Vert
_{H^{1}(\Omega)}\bigr)\Vert v\Vert_{W^{1,s}(\Omega)}. %\label{j1}%
\end{align*}
Therefore, by using (\ref{aux-j}), we obtain
\[
\Bigl|\frac{d}{dt}\int_{\Omega}\rho^{m}u^{m}v\Bigr|\leq h_{m}\Vert
v\Vert_{W^{1,s}(\Omega)},\ \forall v\in\mathcal{D}(\Omega),
\]
for some $h_{m}\in L^{1}(0,T),$ which implies that
\begin{equation}
\Vert\partial_{t}(\rho^{m}u^{m})\Vert_{L^{1}(0,T;W^{-1,s^{\prime}}(\Omega
))}\leq C.\label{aux-deriv2}%
\end{equation}
\newline

\textbf{Step 4: Convergence properties.} From the uniform estimates obtained
in the previous step, we will deduce some convergences for the approximate
solution. We first observe that from \eqref{energyrho}-\eqref{aux-j}, \eqref{aux-deriv-time}, and since $
H^{1}(\Omega)\hookrightarrow\hookrightarrow L^2(\Omega)\hookrightarrow
W^{-1,\frac{2p}{p+1}}(\Omega),$ by applying Lemma 4 of \cite{Simon}, we conclude that
\[
(\rho^{m})\ \mbox{is relatively compact in}\ L^{2}(0,T;L^2(\Omega)).
\]
Similarly, as $(\rho^{m})$ is bounded in $L^{\infty}(0,T;L^{p}(\Omega))$ and
$L^{p}(\Omega)\hookrightarrow\hookrightarrow W^{-1,p}(\Omega)\hookrightarrow
W^{-1,\frac{2p}{p+1}}(\Omega),$ we have that
\[
(\rho^{m})\ \mbox{is relatively compact in}\ C([0,T];W^{-1,p}(\Omega)).
\]
Moreover, as $s>3\geq\frac{6p}{5p-6},$ it holds $L^{\frac{6p}{p+6}}
(\Omega)\hookrightarrow\hookrightarrow W^{-1,\frac{6p}{p+6}}(\Omega
)\hookrightarrow W^{-1,s^{\prime}}(\Omega);$ so, from \eqref{j21} and
\eqref{aux-deriv2}, we get
\[
(\rho^{m}u^{m})\ \mbox{is relatively compact in}\ L^{2}(0,T;W^{-1,\frac
{6p}{p+6}}(\Omega)).
\]

Thus, in view of the uniform bounds obtained in the previous step, we have
that the sequence $(u^{m},\rho^{m},\rho^{m}u^{m},\rho^{m}u^{m}u^{m})$
converges (up to subsequences) to some $(u_{\epsilon},\rho_{\epsilon},\xi
_{1},\xi_{2})$, as $m\rightarrow\infty$, in the following sense:
\begin{align}
\rho^{m} &  \rightarrow\rho_{\epsilon}\ \mbox{strongly}\ \mbox{in}\ L^{2}
(0,T;L^2(\Omega)) \text{ and a.e. in } Q,\label{j012}\\
\rho^{m} &  \rightarrow\rho_{\epsilon}\ \mbox{weakly-}\ast\mbox{in}\ L^{\infty
}(0,T;L^{p}(\Omega)),\label{j012-2}\\
\rho^{m} &  \rightarrow\rho_{\epsilon}\ \mbox{strongly in}\ C([0,T];W^{-1,p}
(\Omega)),\label{j013}\\
u^{m} &  \rightarrow u_{\epsilon}\ \mbox{weakly in}\ L^{2}(0,T;H_{\sigma}%
^{1}(\Omega)),\label{j014}\\
\rho^{m}u^{m} &  \rightarrow\xi_{1}\ \left\{
\begin{array}
[c]{lcl}%
\mbox{weakly in}\ {L^{2}(0,T;L^{\frac{6p}{p+6}}(\Omega))},\label{j015} &  & \\
\mbox{weakly-}\ast\mbox{in}\ {L^{\infty}(0,T;L^{\frac{2p}{p+1}}(\Omega))}, &
& \\
\mbox{strongly in}\ L^{2}(0,T;W^{-1,\frac{6p}{p+6}}(\Omega)), &  &
\end{array}
\right.  \\
\rho^{m}u^{m}u^{m} &  \rightarrow\xi_{2}\ \mbox{weakly in}\
L^{\varsigma /2}(0,T;L^{2}(\Omega)).\label{j017b}%\\
%\sqrt{\rho^{m}}u^{m} &  \rightarrow\xi_{3}\ \mbox{weakly}-*\mbox{
%in}\
%L^{\infty}(0,T;L^{2}(\Omega)).\label{j017b2}
\end{align}

Next, we identify the limits $\xi_{1}$ and $\xi_{2}.$ To
this end, notice that the product mapping from $H^{1}(\Omega)\times
W^{-1,p}(\Omega)$ to $W^{-1,\frac{6p}{p+6}}(\Omega)$ is continuous
(see \cite[Lemma 3]{Simon}). Therefore, \eqref{j013} and
\eqref{j014} lead us to
\begin{equation}
\rho^{m}u^{m}\rightarrow\rho_{\epsilon}u_{\epsilon}\ \mbox{weakly in}\ L^{2}%
(0,T;W^{-1,\frac{6p}{p+6}}(\Omega)). \label{j25}%
\end{equation}
Convergence \eqref{j25} together with \eqref{j015} implies that $\xi_{1}%
=\rho_{\epsilon}u_{\epsilon}.$ In analogous way, from \eqref{j014} and
\eqref{j015}, we obtain that
\[
\rho^{m}u^{m}u^{m}\rightarrow\rho_{\epsilon}u_{\epsilon}u_{\epsilon
}\ \mbox{weakly in}\ L^{1}(0,T;W^{-1,\frac{3p}{p+3}}(\Omega)).
\]
The uniqueness of the limit in the sense of distributions implies
that $\xi_{2}=\rho_{\epsilon}u_{\epsilon}u_{\epsilon}.$
\\
%Finally,
%noting that
%$$\sqrt{\rho^m}u^m-\sqrt{\rho_\epsilon}u_\epsilon=\frac{(\rho^m-\rho_\epsilon)u^m}{(\sqrt{\rho^m}+\sqrt{\rho_\epsilon})}+\sqrt{\rho_\epsilon}(u^m-u_\epsilon),$$
%using that the product mapping from $H^{1}(\Omega)\times
%W^{-1,p}(\Omega)$ to $W^{-1,\frac{6p}{p+6}}(\Omega)$ is continuous,
%and taking into account the convergences \eqref{j013} and
%\eqref{j014}, we can conclude that at least
%\begin{equation}
%\sqrt{\rho^{m}}u^{m}\rightarrow\sqrt{\rho_{\epsilon}}u_{\epsilon}\ \mbox{weakly in}\ L^{2}%
%(0,T;W^{-1,\frac{6p}{p+6}}(\Omega)), \label{j25n}%
%\end{equation}
%as $m\rightarrow\infty$, and therefore $\xi_3=\sqrt{\rho_{\epsilon}}u_{\epsilon}.$

 \textbf{Step 5: Passing to the limit as
$m\rightarrow\infty.$} Finally, we will prove the existence of a
weak-strong solution to problem \eqref{prev1} by passing to the
limit as $m\rightarrow\infty$ in the approximate problem
\eqref{prev2}-\eqref{prev3}.
%we have that
%$\rho^{m}\rightarrow\rho_{\epsilon}$ in $\mathcal{D}^{\prime}(Q),$
%and therefore $\partial_{t}\rho^{m}\rightarrow
%\partial_{t}\rho_{\epsilon}$ in $\mathcal{D}^{\prime}(Q).$ From (\ref{j015})
%we have that $\rho^{m}u^{m}\rightarrow\rho_{\epsilon}u_{\epsilon}$ in
%$\mathcal{D}^{\prime}(Q)$ and thus, $\mbox{div}(\rho^{m}u^{m})\rightarrow
%\mbox{div}(\rho_{\epsilon}u_{\epsilon})$ in $\mathcal{D}^{\prime}(Q).$
%Moreover, $\epsilon\Delta\rho^{m}\rightarrow\epsilon\Delta\rho_{\epsilon}$ in
%$\mathcal{D}^{\prime}(Q).$

Since $\rho^{m}\rightarrow\rho_{\epsilon}$, $\rho^{m}u^{m}\rightarrow
\rho_{\epsilon}u_{\epsilon}$ and $\Delta\rho^{m}\rightarrow\Delta
\rho_{\epsilon}$ in $\mathcal{D}^{\prime}(Q),$ we can pass to the limit in
\eqref{prev2}$_{1}$ in the sense of distributions. Therefore
\[
\partial_{t}\rho_{\epsilon}+\mbox{div}\ (\rho_{\epsilon}u_{\epsilon}%
)=\epsilon\Delta\rho_{\epsilon}\ \mbox{in}\ \ \mathcal{D}^{\prime}(Q).
\]
In view of the estimate (\ref{aux-reg-rho1}), and since $\rho_{\epsilon}$
inherits regularity from $\rho^{m}$, we have in particular that
\[
\partial_{t}\rho_{\epsilon}+\mbox{div}\ (\rho_{\epsilon}u_{\epsilon}%
)=\epsilon\Delta\rho_{\epsilon}\ \text{a.e. in } Q.
%\text{ in }L^{2}(0,T;W^{-1,\frac{2p}{p+1}}(\Omega)).
\]

From \eqref{j013} we have that $\rho^{m}(0)\rightarrow\rho_{\epsilon}(0)$ in
$W^{-1,p}(\Omega).$ But, the initial condition \eqref{prev2}$_{3}$ says that
$\rho^{m}(0)= \rho_{0,\epsilon}\in C^{2,r}(\bar{\Omega})$, so that
$\rho_{\epsilon}(0)=\rho_{0,\epsilon}.$ It is not difficult to see that the
boundary condition is also satisfied.

In order to pass to the limit in \eqref{prev3}$_{1}$ as $m\rightarrow\infty$, we notice that
by using the energy identity \eqref{energyrho} for $ \rho^m $ and the corresponding one for $\rho^\epsilon$,
we can prove that $\nabla \rho^m $ converges strongly to $ \nabla \rho^\epsilon $ in $L^2(0,T;L^2(\Omega)),$
(see \cite{Feireisl2001} for a similar argument).
Hence,  from the convergences obtained in step 4, we can classically pass to the limit
in \eqref{prev3}$_{1}$ as $m\rightarrow\infty$, and obtain that \eqref{aprox1}
holds true.
The proof is complete.
\end{proof}

\bigskip

\subsection*{Proof Theorem \ref{teorem1}.}

\begin{proof}
Let $(\rho_{\epsilon},u_{\epsilon})$ be the weak-strong solution of
\eqref{prev1} given by Proposition \ref{prop1}. As we mentioned in step 3 of
the proof of Proposition \ref{prop1}, most of the obtained estimates are also
independent of $\epsilon>0.$ So, proceeding as in step 4, we can conclude the
existence of $(\rho,u,\beta_{1},\beta_{2})$ and a subsequence of
$(\rho_{\epsilon},u_{\epsilon},\rho_{\epsilon}u_{\epsilon},\rho_{\epsilon
}u_{\epsilon}u_{\epsilon})$, still indexed by $\epsilon,$ such that the
following convergences hold true as $\epsilon\rightarrow0$:
\begin{align}
\rho_{\epsilon}  &  \rightarrow\rho\ \mbox{weakly}-\ast\ \mbox{in}\ {L^{\infty
}(0,T;L^{p}(\Omega))},\label{j012c}\\
\rho_{\epsilon}  &  \rightarrow\rho\ \mbox{strongly in}\ C([0,T];W^{-1,p}%
(\Omega)),\label{j013c}\\
u_{\epsilon}  &  \rightarrow u\ \mbox{weakly in}\ {L^{2}(0,T;H_{\sigma}%
^{1}(\Omega))},\label{j014c}\\
\rho_{\epsilon}u_{\epsilon}  &  \rightarrow\beta_{1}\ \left\{
\begin{array}
[c]{lcl}%
\mbox{weakly in}\ {L^{2}(0,T;L^{\frac{6p}{p+6}}(\Omega))},\label{j015c} &  &
\\
\mbox{weakly}-\ast\ \mbox{in}\ {L^{\infty}(0,T;L^{\frac{2p}{p+1}}(\Omega))}, &
& \\
\mbox{strongly in}\ L^{2}(0,T;W^{-1,\frac{6p}{p+6}}(\Omega)), &  &
\end{array}
\right. \\
\rho_{\epsilon}u_{\epsilon}u_{\epsilon}  &  \rightarrow\beta_{2}%
\ \mbox{weakly in}\ {L^{\varsigma/2}(0,T;L^{2}(\Omega))}. \label{j017bc}
%\\
%\sqrt{\rho_{\epsilon}}u_{\epsilon}  &  \rightarrow\beta_{3}%
%\ \mbox{weakly in}\ {L^{\infty}(0,T;L^{2}(\Omega))}. \label{j017bc1}
\end{align}
Moreover, using (\ref{aux-j}) and the continuity of the product mapping from
$H^{1}(\Omega)\times W^{-1,2}(\Omega)$ to $W^{-1,\frac{3}{2}}(\Omega)$, we
obtain
\begin{align*}
\epsilon\nabla\rho_{\epsilon}  &  \rightarrow0\ \mbox{strongly in}\ {L^{2}%
(0,T;L^{2}(\Omega))},\\
\epsilon\,u_{\epsilon}\Delta\rho_{\epsilon}  &  \rightarrow
0\ \mbox{weakly in}\ {L^{1}(0,T;W^{-1,\frac{3}{2}}(\Omega))}.
\end{align*}

Working as in the end of step 3, one can prove that $\beta_{1}=\rho
u$ and $\beta_{2}=\rho uu$. Moreover, from
above convergences, in analogous way as we did in step 5, we can
pass to the limit in the regularized problem \eqref{prev1}. The only
difference here is that the two terms involving the parameter
$\epsilon$ vanish.

%We only observe the following fact:
%$\rho_{\epsilon}\rightarrow\rho$ in $\mathcal{D}^{\prime}(Q),$ and
%therefore $\partial_{t}\rho_{\epsilon }\rightarrow\partial_{t}\rho$
%in $\mathcal{D}^{\prime}(Q).$ Furthermore, from \eqref{j015c}
%we have that $\rho_{\epsilon}u_{\epsilon}\rightarrow\rho u$ in $\mathcal{D}
%^{\prime}(Q)$ and thus, $\mbox{div}(\rho_{\epsilon}u_{\epsilon})\rightarrow
%\mbox{div}(\rho u)$ in $\mathcal{D}^{\prime}(Q).$ Moreover, $\epsilon
%\Delta\rho^{m}\rightarrow0$ in $\mathcal{D}^{\prime}(Q).$ Therefore
%\[
%\partial_{t}\rho+\mbox{div}\ (\rho u)=0\ \mbox{in}\ \ \mathcal{D}^{\prime
%}(Q).
%\]
Thenceforth, we obtain that equation $\partial_{t}\rho+\mbox{div}\ (\rho u)=0$ is verified in
$\mathcal{D}^{\prime}(Q).$ From \eqref{j013c} we have that $\rho_{0,\epsilon}=\rho_{\epsilon
}(0)\rightarrow\rho(0)$ in $W^{-1,p}(\Omega).$ However, as $\rho_{0,\epsilon
}\rightarrow\rho_{0}$ in $L^{p}(\Omega),$ when $6\leq p<\infty,$ and
$\rho_{0,\epsilon}\rightharpoonup\rho_{0}$ weakly-$\ast$ in $L^{\infty}
(\Omega),$ when $p=\infty,$ then $\rho(0)=\rho_{0}$ at least in $W^{-1,p}
(\Omega).$
%On the other hand, in the same spirit of the proof of
%Proposition \ref{prop1} (see step 5, proof of the condition
%\emph{(iii)} of Definition \ref{defaprox}), one can prove that
%\[
%\int_{\Omega}(\rho u)(0)\varphi dx=\int_{\Omega}v_{0}\varphi dx,\ \forall
%\ \varphi\in H_{\sigma}^{1}(\Omega).
%\]

Finally, it remains to verify the energy inequality (\ref{energyineq}).
%For
%that, notice that $(\rho_{\epsilon},u_{\epsilon})$ verifies the energy
%identity (see (\ref{e2aaaa})):
%\begin{equation}
%\frac{1}{2}\frac{d}{dt}\Vert\sqrt{\rho_{\epsilon}}u_{\epsilon}(t)\Vert_{2}%
%^{2}+2\nu\Vert Du_{\epsilon}(t)\Vert_{2}^{2}+2\nu\alpha\int_{\partial\Omega
%}|u_{\epsilon}|^{2}=\int_{\Omega}\rho_{\epsilon}f\cdot u_{\epsilon}.
%\label{en2}%
%\end{equation}
As usual, integrating \eqref{e2aaaa} over $(0,t),$ multiplying the result by $\phi
\in\mathcal{D}(0,T),\ \phi\geq 0,$ integrating over $(0,T),$ and passing to the limit using convergences \eqref{j012}-\eqref{j017b}, we arrive at the following inequality
\begin{align*}
&  \int_{0}^{T}\left(  \frac{1}{2}\Vert\rho_{\epsilon}|u_{\epsilon
}|^2(t)\Vert_{1}+\int_{0}^{t}[2\nu\Vert Du_{\epsilon}(s)\Vert_{2}^{2}
+2\nu\alpha\int_{\partial\Omega}|u_{\epsilon}|^{2}]ds\right)  \phi
(t)dt\nonumber\\
&  \leq {\frac{1}{2}}\int_{\Omega}\frac{|v_{0,\epsilon}|^{2}}{\rho_{0,\epsilon}}dx\int_{0}
^{T}\phi(t)dt+\int_{0}^{T}\left(  \int_{0}^{t}\int_{\Omega}\rho_{\epsilon
}f\cdot u_{\epsilon}\right)  \phi(t)dt.
\end{align*}
As $\epsilon\rightarrow0,$ we have $\int_{\Omega}\frac{|v_{0,\epsilon}|^{2}
}{\rho_{0,\epsilon}}dx\rightarrow\int_{\Omega}\frac{|v_{0}|^{2}}{\rho_{0}
}dx\leq C.$ From convergences (\ref{j012c})-(\ref{j017bc}), as $\epsilon
\rightarrow0$ we find
\begin{align*}
&  \int_{0}^{T}\left(  \frac{1}{2}\Vert \rho|u|^2(t)\Vert_{1}+\int
_{0}^{t}[2\nu\Vert Du(s)\Vert_{2}^{2}+2\nu\alpha\int_{\partial\Omega}
|u|^{2}]ds\right)  \phi(t)dt\\ &  \leq {\frac{1}{2}}\int_{\Omega}\frac{|v_{0}|^{2}}{\rho_{0}}dx\int_{0}^{T}\phi(t)dt
+\int_{0}^{T}\left(  \int_{0}^{t}\int_{\Omega}\rho f\cdot u\right)  \phi(t)
dt,
\end{align*}
for any $\phi\in\mathcal{D}(0,T),\ \phi\geq0.$ This yields the energy
inequality (\ref{energyineq}) which implies in particular that $\sqrt{\rho} u \in L^\infty(0,T;L^2(\Omega)).$
\end{proof}

\vspace{.2cm}

\begin{remark}
Observe that, from the regularity of $\rho$ and $\rho u$, equation
$\partial_{t}\rho+\mbox{div}\ (\rho u)=0$ holds in
\[
W^{-1,\infty}(0,T;L^{p}(\Omega))\cap L^{\infty}(0,T;W^{-1,\frac{2p}{p+1}%
}(\Omega)).
\]

\end{remark}

\begin{remark}
Concerning the pressure, observe that
\begin{align*}
\rho u  &  \in L^{\infty}(0,T;L^{\frac{2p}{p+1}}(\Omega))\Rightarrow
\partial_{t}(\rho u)\in W^{-1,\infty}(0,T;L^{\frac{2p}{p+1}}(\Omega)),\\
\rho uu  &  \in L^{\varsigma/2}(0,T;L^{2}(\Omega))\Rightarrow\mbox{div}(\rho
uu)\in L^{\varsigma/2}(0,T;W^{-1,2}(\Omega)),\\
u  &  \in L^{2}(0,T;H_{\sigma}^{1}(\Omega))\Rightarrow\Delta u\in
L^{2}(0,T;W^{-1,2}(\Omega)),\\
\rho f  &  \in L^{1}(0,T;L^{\frac{2p}{p+1}}(\Omega))\Rightarrow\rho f\in
L^{1}(0,T;W^{-1,\frac{2p}{p+1}}(\Omega)).
\end{align*}
From the \textit{De Rham theorem} there exists a distribution $\pi\in
W^{-1,\infty}(0,T;L^{\frac{2p}{p+1}}(\Omega))$ such that
\[
\partial_{t}(\rho u)+\mbox{div}(\rho uu)-\nu\Delta u+\nabla\pi=\rho
f\ \mbox{in}\ W^{-1,\infty}(0,T;W^{-1,\frac{2p}{p+1}}(\Omega)).
\]

\end{remark}

\section{Vanishing viscosity limit}

\label{inviscid}

In this section we establish the convergence of a weak solution of the
non-homogeneous Navier-Stokes equations with Navier boundary conditions to the
strong solution of non-homogeneous Euler equations when the viscosity
coefficient goes to zero. To this end, consider the following limiting
problem
\begin{equation}
\left\{
\begin{array}
[c]{rcl}%
\partial_{t}(\rho u)+\mbox{div}\ (\rho uu)+\nabla\pi=\rho f\  &
\mbox{in}\ Q, & \\
\mbox{div}\ u=0\  & \mbox{in}\ Q, & \\
\partial_{t}\rho+\mbox{div}\ (\rho u)=0\  & \mbox{in}\ Q, & \\
u\cdot n=0\  & \mbox{on}\ \Sigma, & \\
\rho(0)=\rho_{0}\  & \mbox{in}\ \Omega, & \\
u(0)=u_{0}\  & \mbox{in}\ \Omega. &
\end{array}
\right.  \label{eulernonhom}%
\end{equation}

We recall a result of existence of strong solutions to the nonhomogeneous
Euler system \eqref{eulernonhom}. For another results on this subject see
\cite{itoh-tani,beirao, marsden}.

\begin{theorem}
\label{euler1}(see \cite{valli88}) Let $\Omega$ with boundary $\partial\Omega$
smooth enough, and $p>3$. Assume that $\rho_{0}\in W^{2,p}(\Omega)$,
$0<\rho_{\ast}\leq\rho_{0}\leq\rho^{\ast}<\infty$, for some constants
$\rho_{\ast},\rho^{\ast};$ moreover, assume that $u_{0}\in W^{2,p}(\Omega)$,
$u_{0}\cdot n=0$ on $\partial\Omega$, $div(u_{0})=0$, $f\in L^{1}%
(0,T;W^{2,p}(\Omega))$. Then there exists a time $T_{\ast}\in(0,T]$ such that
problem \eqref{eulernonhom} has a unique solution $(\rho,u,\pi)$ which satisfies

\begin{itemize}
\item[a)] $\rho\in C([0,T_{*}];W^{2,p}(\Omega)) \cap C^{1}([0, T_{*}%
];W^{1,p}(\Omega))$, $0<\rho_{*}\leq\rho\leq\rho^{*}<\infty$,

\item[b)] $u\in C([0, T_{*}];W^{2,p}(\Omega)) \cap W^{1,1}(0, T_{*}%
;W^{1,p}(\Omega))$,

\item[c)] $\pi\in L^{1}(0, T_{*};W^{3,p}(\Omega))$.
\end{itemize}
Moreover, if $f\in C([0, T];W^{2,p}(\Omega))$ then $u\in C^{1}([0,
T_{*}];W^{1,p}(\Omega))$ and $\pi\in C([0, T_{*}];W^{3,p}(\Omega))$, hence
$(\rho, u,\pi)$ is a classical solution.
\end{theorem}

\begin{remark}
\label{regeuler} Observe that, since $p > 3,$ by the Sobolev embedding we have
that $\nabla\rho\in L^{\infty}(\Omega\times(0,T_{*})),\ \nabla u \in
L^{\infty}(\Omega\times(0,T_{*})),\ u_{t} \in L^{1}(0,T_{*};L^{\infty}%
(\Omega)) $ and $\nabla\pi\in L^{1}(0,T_{*};L^{\infty}(\Omega)) $.
\end{remark}

Let $(u^{\nu},\rho^{\nu})$ be a weak solution of
\begin{equation}
\left\{
\begin{array}
[c]{rcl}%
\partial_{t}(\rho^{\nu}u^{\nu})+\mbox{div}\ (\rho^{\nu}u^{\nu}u^{\nu}%
)-\nu\Delta u^{\nu}+\nabla\pi^{\nu}=\rho^{\nu}f^{\nu}\  & \mbox{in}\ Q, & \\
\mbox{div}\ u^{\nu}=0\  & \mbox{in}\ Q, & \\
\partial_{t}\rho^{\nu}+\mbox{div}\ (\rho^{\nu}u^{\nu})=0\  & \mbox{in}\ Q, &
\\
u^{\nu}\cdot n=0\  & \mbox{on}\ \Sigma, & \\
\left[  D(u^{\nu})n+\alpha u^{\nu}\right]  _{tan}=0\  & \mbox{on}\ \Sigma, &
\\
\rho^{\nu}(0)=\rho_{0}^{\nu}\  & \mbox{in}\ \Omega, & \\
(\rho^{\nu}u^{\nu})(0)=v_{0}^{\nu}\  & \mbox{in}\ \Omega, &
\end{array}
\right.  \label{aux-v}%
\end{equation}
with $0<\rho_{0}^{\nu}\in L^{p}(\Omega)$ given by Theorem \ref{teorem1}. This
weak solution satisfies
\begin{multline}
\frac{1}{2}\Vert\sqrt{\rho^{\nu}(t)}u^{\nu}(t)\Vert_{2}^{2}+2\nu\alpha\int
_{0}^{t}\int_{\partial\Omega}|u^{\nu}|^{2}+2\nu\int_{0}^{t}\Vert D(u^{\nu
})\Vert_{2}^{2}\label{estenergiau}\\
\leq\frac{1}{2}\Bigl\Vert\frac{v_{0}^{\nu}}{\sqrt{\rho_{0}^{\nu}}}\Bigr
\Vert_{2}^{2}+\int_{0}^{t}\int_{\Omega}\rho^{\nu}f^{\nu}\cdot u^{\nu},
\end{multline}
and
\begin{equation}
\frac{1}{2}\Vert\rho^{\nu}(t)\Vert_{2}^{2}\leq\frac{1}{2}\Vert\rho_{0}^{\nu
}\Vert_{2}^{2}. \label{estenergiarho}
\end{equation}

\vspace{.2cm}

We now state the main result of this section.

\begin{theorem}
\label{limitinv} Under the hypotheses of Theorems \ref{euler1} and
\ref{teorem1}. Let $(u,\rho)$ be the solution of (\ref{eulernonhom}) obtained
in Theorem \ref{euler1}, and let $(u^{\nu},\rho^{\nu})$ be the one of
(\ref{aux-v}) given in Theorem \ref{teorem1}. Assume further that
$0<\rho_{\ast}\leq\rho_{0}^{\nu}\leq\rho^{\ast},$ for all $\nu>0,$ where
$\rho_{\ast}$ and $\rho^{\ast}$ are the same constants of Theorem
\ref{euler1}. Then there exists $C>0$ independent of $\nu>0$ such that, for
all $t\in(0,T_{\ast}],$ the following inequality holds:
\begin{multline*}
\Vert u(t)-u^{\nu}(t)\Vert_{2}^{2}+\Vert\rho(t)-\rho^{\nu}(t)\Vert_{2}^{2}\\
\leq C\Bigl(\Bigl\Vert\sqrt{\rho_{0}^{\nu}}u_{0}-\frac{v_{0}^{\nu}}{\sqrt
{\rho_{0}^{\nu}}}\Bigr\Vert_{2}^{2}+\Vert\rho_{0}-\rho_{0}^{\nu}\Vert_{2}%
^{2}+\nu\int_{0}^{T_{\ast}}\Vert u\Vert_{H^{1}(\Omega)}^{2}+\int_{0}^{T_{\ast
}}\Vert f-f^{\nu}\Vert_{\frac{2p}{p-1}}\Bigr).
\end{multline*}
In particular, if $\Bigl\Vert\sqrt{\rho_{0}^{\nu}}u_{0}-\frac{v_{0}^{\nu}%
}{\sqrt{\rho_{0}^{\nu}}}\Bigr\Vert_{2}^{2}+\Vert\rho_{0}-\rho_{0}^{\nu}%
\Vert_{2}^{2}\rightarrow0$ as $\nu\rightarrow0,$ and $f^{\nu}$ converges to
$f$ in $L^{1}(0,T_{\ast};L^{\frac{2p}{p-1}}(\Omega)),$ then
\begin{equation}
\label{convenergia}\sup_{0<s<t}(\Vert u(s)-u^{\nu}(s)\Vert_{2}+\Vert
\rho(s)-\rho^{\nu}(s)\Vert_{2})\rightarrow0\text{ as }\nu\rightarrow0.
\end{equation}

\end{theorem}

\vspace{.2cm}

\begin{remark}
We stress that \eqref{convenergia} implies in particular that
\[
\sup_{0<s<t}\Vert\sqrt{\rho} u- \sqrt{\rho^{\nu}} u^{\nu}\Vert_{2}%
\rightarrow0\text{ as }\nu\rightarrow0,
\]
since we can write
\[
\sqrt{\rho} u- \sqrt{\rho^{\nu}} u^{\nu} = (\sqrt{\rho} - \sqrt{\rho^{\nu}})u +
\sqrt{\rho^{\nu}}(u - u^{\nu}) .
\]
Moreover, by interpolation, we conclude that, for $p \geq2, $
\[
\sup_{0<s<t} \Vert\rho-\rho^{\nu}\Vert_{p}\rightarrow0\text{ as }%
\nu\rightarrow0.
\]
\end{remark}

Let us proceed with the proof of Theorem \ref{limitinv}.

\begin{proof}
The differences $\omega=u-u^{\nu},\ \sigma=\rho-\rho^{\nu},\ q=\pi-\pi^{\nu}$
satisfy

\begin{equation}
\left\{
\begin{array}
[c]{rcl}%
\rho^{\nu}[\omega_{t}+u^{\nu}\cdot\nabla\omega]+\nabla q &  & \\
=-\rho^{\nu}\omega\cdot\nabla u-\sigma(u_{t}+u\cdot\nabla u)-\nu\Delta u^{\nu
}+\sigma f+\rho^{\nu}(f-f^{\nu})\  & \mbox{in}\ Q, & \\
\mbox{div}\ \omega=0\  & \mbox{in}\ Q, & \\
\partial_{t}\sigma+u^{\nu}\cdot\nabla\sigma=-\omega\cdot\nabla\rho\  &
\mbox{in}\ Q, & \\
\omega\cdot n=0\  & \mbox{on}\ \Sigma, & \\
\sigma(0)=\rho_{0}-\rho_{0}^{\nu}\  & \mbox{in}\ \Omega, & \\
\omega(0)=u_{0}-u_{0}^{\nu}\  & \mbox{in}\ \Omega. &
\end{array}
\right.  \label{difference2}
\end{equation}
Formally, by multiplying third equation in (\ref{difference2}) by $\sigma$ and
integrating in space and time we obtain

\begin{equation}
\frac{1}{2}\Vert\sigma(t)\Vert_{2}^{2}\leq\frac{1}{2}\Vert\sigma(0)\Vert
_{2}^{2}-\int_{0}^{t}\int_{\Omega}\omega\cdot\nabla\rho\,\sigma,
\label{presigma}%
\end{equation}
where we have used that $div(u^{\nu})=0$ and $u^{\nu}\cdot n=0$ on
$\partial\Omega$ to deduce that $\int_{0}^{t}\int_{\Omega}\left(  u^{\nu}%
\cdot\nabla\sigma\right)  \sigma=0$. In fact, to justify the previous
inequality we proceed in this way: consider the energy inequality
\eqref{estenergiarho} and obtain other three inequalities: $J_{1}$ by
multiplying the continuity equation \eqref{eulernonhom}$_{3}$ for $\rho$ by
$\rho$, $J_{2}$ by multiplying the same equation \eqref{eulernonhom}$_{3}$ by
$\rho^{\nu}$ and $J_{3}$ by multiplying the equation \eqref{aux-v}$_{3}$ for
$\rho^{\nu}$ by $\rho$. One obtains \eqref{presigma} by doing
$\eqref{estenergiarho}+J_{1}-J_{2}-J_{3}$.

From \eqref{presigma}, by using the H\"older and Young inequalities
we arrive at
\begin{equation}
\label{estsigma}\|\sigma(t) \|_{2}^{2} \leq\|\sigma(0) \|_{2}^{2} + \int
_{0}^{t} \|\nabla\rho\|_{\infty} (\|\omega\|_{2}^{2} +\|\sigma\|_{2}^{2}).
\end{equation}

Formally, by multiplying the first equation in (\ref{difference2}) by $\omega
$, and integrating in space and time, and finally by using integration by
parts (Lemma \ref{part}) in the Laplacian term, one gets
\begin{multline}
\frac{1}{2}\Vert\sqrt{\rho^{\nu}(t)}\omega(t)\Vert_{2}^{2}+2\nu\alpha\int
_{0}^{t}\int_{\partial\Omega}u^{\nu}\cdot w+2\nu\int_{0}^{t}\int_{\Omega
}D(u^{\nu})\cdot D(w)\label{energyw}\\
\leq\frac{1}{2}\Bigl\Vert\sqrt{\rho_{0}^{\nu}}u_{0}-\frac{v_{0}^{\nu}}{\sqrt
{\rho_{0}^{\nu}}}\Bigr\Vert_{2}^{2}-\int_{0}^{t}\int_{\Omega}\rho^{\nu}\omega
\cdot\nabla u\cdot\omega+\int_{0}^{t}\int_{\Omega}\sigma(u_{t}+u\cdot\nabla
u)\cdot\omega\\
+\int_{0}^{t}\int_{\Omega}\sigma f\cdot\omega+\int_{0}^{t}\int_{\Omega}%
\rho^{\nu}(f-f^{\nu})\cdot\omega.
\end{multline}
To justify the previous inequality, we first obtain three inequalities:
$I_{1}$ by multiplying the momentum equation \eqref{eulernonhom}$_{1}$ by $u$,
$I_{2}$ by multiplying the same equation \eqref{eulernonhom}$_{1}$ by $u^{\nu
}$ and $I_{3}$ by multiplying the momentum equation \eqref{aux-v}$_{1}$ for
$u^{\nu}$ by $u$. Next, we consider the energy inequality \eqref{estenergiau}
and do $\eqref{estenergiau}+I_{1}-I_{2}-I_{3}$ to arrive at \eqref{energyw}.

Proceeding as in \cite{Iftimie-Planas}, we rewrite
\begin{multline*}
2\nu\alpha\int_{0}^{t}\int_{\partial\Omega}u^{\nu}\cdot w+2\nu\int_{0}^{t}
\int_{\Omega}D(u^{\nu})\cdot D(w)\\
=2\nu\alpha\int_{0}^{t}\int_{\partial\Omega}\bigl|w+\frac{u}{2}\bigr|^{2}
+2\nu\int_{0}^{t}\int_{\Omega}\bigl|D\bigl(w+\frac{u}{2}\bigr)\bigr|^{2}
-\frac{\nu\alpha}{2}\int_{0}^{t}\int_{\partial\Omega}|u|^{2}-\frac{\nu}{2}
\int_{0}^{t}\int_{\Omega}|D(u)|^{2}.
\end{multline*}
Thus, by using H\"{o}lder and Young inequalities, we have
\begin{align*}
\frac{1}{2}\Vert\sqrt{\rho^{\nu}}\omega\Vert_{2}^{2}  &  \leq\frac{1}{2}
\Bigl\Vert\sqrt{\rho_{0}^{\nu}}u_{0}-\frac{v_{0}^{\nu}}{\sqrt{\rho_{0}^{\nu}}}
\Bigr\Vert_{2}^{2}+\int_{0}^{t}\Vert\nabla u\Vert_{\infty}\Vert\sqrt{\rho^{\nu}
}\omega\Vert_{2}^{2}+\int_{0}^{t}\Vert u_{t}\Vert_{\infty}\Vert\sigma\Vert
_{2}\Vert\omega\Vert_{2}\\
&  +\int_{0}^{t}\Vert u\cdot\nabla u\Vert_{\infty}\Vert\sigma\Vert_{2}
\Vert\omega\Vert_{2}+\frac{\nu\alpha}{2}\int_{0}^{t}\int_{\partial\Omega
}|u|^{2}+\frac{\nu}{2}\int_{0}^{t}\Vert D(u)\Vert_{2}^{2}\\
&  +\int_{0}^{t}\Vert f\Vert_{\infty}\Vert\sigma\Vert_{2}\Vert\omega\Vert
_{2}+\int_{0}^{t}\Vert\sqrt{\rho^{\nu}}\Vert_{2p}\Vert f-f^{\nu}\Vert
_{\frac{2p}{p-1}}\Vert\sqrt{\rho^{\nu}}\omega\Vert_{2}\\
&  \leq\frac{1}{2}\Bigl\Vert\sqrt{\rho_{0}^{\nu}}u_{0}-\frac{v_{0}^{\nu}}
{\sqrt{\rho_{0}^{\nu}}}\Bigr\Vert_{2}^{2}+\int_{0}^{t}\Vert\nabla u\Vert_{\infty
}\Vert\sqrt{\rho^{\nu}}\omega\Vert_{2}^{2}\\
&  +\frac{1}{2}\int_{0}^{t}(\Vert u_{t}\Vert_{\infty}+\Vert u\cdot\nabla
u\Vert_{\infty}+\Vert f\Vert_{\infty})(\Vert\omega\Vert_{2}^{2}+\Vert
\sigma\Vert_{2}^{2})\\
&  +C\nu\int_{0}^{t}\Vert u\Vert_{H^{1}(\Omega)}^{2}+C\Vert\sqrt{\rho^{\nu}
}\Vert_{L^{\infty}(0,T_{\ast};L^{2p}(\Omega))}\int_{0}^{t}\Vert f-f^{\nu}
\Vert_{\frac{2p}{p-1}}\Vert\sqrt{\rho^{\nu}}\omega\Vert_{2}%
\end{align*}
where $C=C(\rho^{\ast},\alpha,\Omega)$ is independent of $\nu$ and we have
used that $\displaystyle\int_{\partial\Omega}|u|^{2}\leq C(\Omega)\Vert
u\Vert_{H^{1}(\Omega)}.$ Notice also that
\begin{equation}
\Vert\sqrt{\rho^{\nu}}\Vert_{L^{\infty}(0,T_{\ast};L^{\infty}(\Omega))}
=\Vert\rho^{\nu}\Vert_{L^{\infty}(0,T_{\ast};L^{\infty}(\Omega))}^{1/2}
\leq\Vert\rho_{0}^{\nu}\Vert_{\infty}^{1/2}\leq C(\rho^{\ast})^{1/2}.
\label{aux-est-1}%
\end{equation}
Adding \eqref{estsigma} and using (\ref{aux-est-1}), we obtain
\begin{multline*}
\sqrt{\rho_{\ast}}\Vert\omega(t)\Vert_{2}^{2}+\Vert\sigma(t)\Vert_{2}^{2}\\
\leq \Bigl \Vert\sqrt{\rho_{0}^{\nu}}u_{0}-\frac{v_{0}^{\nu}}{\sqrt{\rho_{0}^{\nu}}
}\Bigr\Vert_{2}^{2}+\Vert\sigma(0)\Vert_{2}^{2}+C\nu\int_{0}^{t}\Vert u\Vert
_{H^{1}(\Omega)}^{2}+C\int_{0}^{t}\Vert f-f^{\nu}\Vert_{\frac{2p}{p-1}}
\Vert\omega\Vert_{2}\\
+C\int_{0}^{t}(\Vert\nabla\rho\Vert_{\infty}+\Vert\nabla u\Vert_{\infty}+\Vert
u_{t}\Vert_{\infty}+\Vert u\cdot\nabla u\Vert_{\infty}+\Vert f\Vert_{\infty
})(\Vert\omega\Vert_{2}^{2}+\Vert\sigma\Vert_{2}^{2}).
\end{multline*}
Now we apply a Gronwall type inequality \cite[p.360]{mitri} to obtain
\begin{multline*}
\Vert u(t)-u^{\nu}(t)\Vert_{2}^{2}+\Vert\rho(t)-\rho^{\nu}(t)\Vert_{2}^{2}\\
\leq C\Bigl(\Bigl\Vert\sqrt{\rho_{0}^{\nu}}u_{0}-\frac{v_{0}^{\nu}}{\sqrt{\rho
_{0}^{\nu}}}\Bigr\Vert_{2}^{2}+\Vert\rho_{0}-\rho_{0}^{\nu}\Vert_{2}^{2}+\nu
\int_{0}^{T_{\ast}}\Vert u\Vert_{H^{1}(\Omega)}^{2}+\int_{0}^{T_{\ast}}\Vert
f-f^{\nu}\Vert_{\frac{2p}{p-1}}\Bigr),
\end{multline*}
where $C$ is a positive constant depending on
\begin{align*}
&  \rho_{\ast},\rho^{\ast},\alpha,\Omega,\Vert\nabla\rho\Vert_{{L^{1}%
(0,T_{\ast};L^{\infty}(\Omega))}},\ \Vert\nabla u\Vert_{L^{1}(0,T_{\ast
};L^{\infty}(\Omega))},\ \\
&  \Vert u_{t}\Vert_{L^{1}(0,T_{\ast};L^{\infty}(\Omega))},\linebreak\Vert
u\Vert_{L^{\infty}(\Omega\times(0,T_{\ast}))}\text{ and }\Vert f\Vert
_{L^{1}(0,T_{\ast};L^{\infty}(\Omega))}.
\end{align*}
Due to regularity of $\rho,u$ and $f$ in Theorem \ref{euler1} (see Remark
\ref{regeuler}), all the above norms  are
finite. The proof is then complete.
\end{proof}

\end{document}